\documentclass[reqno,a4paper]{amsart}

\usepackage[italian,english]{babel}
\usepackage{eucal,amsfonts,amssymb,amsmath,amsthm,epsfig,mathrsfs}
\usepackage{color}
\usepackage[dvipsnames]{xcolor}
\usepackage{amscd,amsxtra}
\usepackage{enumerate}
\usepackage{latexsym}
\usepackage{bm}
\usepackage{tcolorbox}
\usepackage{mathtools}

\newtheorem{thm}{Theorem}[section]
\newtheorem{lemma}[thm]{Lemma}
\newtheorem{pro}[thm]{Proposition}
\newtheorem{cor}[thm]{Corollary}
\theoremstyle{definition}
\newtheorem{defn}[thm]{Definition}
\newtheorem{example}[thm]{Example}

\newtheorem{hyp}[thm]{Hypothesis}
\newtheorem*{rmk}{Remark}{\rm}

\theoremstyle{remark}

\numberwithin{equation}{section}

\renewcommand{\hat}[1]{\widehat{#1}}
\renewcommand{\tilde}[1]{\widetilde{#1}}

\newcommand{\gen}[1]{{\left\langle #1\right\rangle}}

\newcommand{\norm}[1]{{\left\|#1\right\|}}

\newcommand{\ra}{\rightarrow}

\newcommand{\N}{\mathbb{N}}

\newcommand{\R}{\mathbb{R}}
\newcommand{\C}{\mathbb{C}}

\newcommand{\eps}{\varepsilon}

\allowdisplaybreaks

\begin{document}

\frenchspacing

\title[Hypercontractivity type property for generalized Mehler semigroups]{Hypercontractivity type property for generalized Mehler semigroups}
\author[L. Angiuli]{Luciana Angiuli}

\author[S. Ferrari]{Simone Ferrari$^*$}
\thanks{$^*$Corresponding author}

\address{L.A., S.F.:  Dipartimento di Matematica e Fisica ``Ennio De Giorgi'', Universit\`a del Salento, and I.N.F.N., Sezione di Lecce, Via per Arnesano, 73100 LECCE, Italy}
\email{\textcolor[rgb]{0.00,0.00,0.84}{{luciana.angiuli@unisalento.it}, {simone.ferrari@unisalento.it}}}
%



\date{\today}

\begin{abstract}
We investigate the hypercontractivity property of generalized Mehler semigroups on the $L^p$-scale with respect to invariant measures. This property is first obtained in the purely theoretical setting of skew operators and, subsequently, deduced for generalized Mehler semigroups arising from linear stochastic differential equations perturbed by L\'evy noise. When the associated invariant measure $\mu$ lacks a purely Gaussian structure, jump components may prevent the validity of Nelson's classical $L^p$-$L^q$ estimates. However, a summability-improving property can be obtained in the setting of mixed-norm spaces $\mathcal{X}_{p,q}(E;\gamma,\pi)$ related to the factorization of the invariant measure $\mu = \gamma * \pi$ into a Gaussian part $\gamma$ and an infinitely divisible non-Gaussian part $\pi$. As in the classical Gaussian case, some modified logarithmic Sobolev inequalities with respect to invariant measures can be inferred.

 
\end{abstract}

\maketitle


\section{Introduction}

This paper is devoted to hypercontractivity results for solutions to evolution equations driven by a class of differential and pseudo-differential operators $\mathcal L$, both in finite and infinite dimension. 
The operators $\mathcal L$ we are interested in, are, in some sense, generators of the so-called \emph{generalized Mehler semigroups}.
Generalized Mehler semigroups are defined for real-valued, bounded and Borel measurable functions $f:E\ra\R$, i.e. $f\in B_b(E)$, by the formula
\begin{equation}\label{GMsgrp}
[P(t)f](x)=\int_E f(S(t) x+y) d\mu_t(y),
\end{equation}
where $E$ is a (finite or infinite-dimensional) Hilbert space, $(S(t))_{t\geq 0}$ is a strongly
continuous semigroup of bounded operators on $E$ and $(\mu_t)_{t\geq 0}$
is a family of Borel probability measures on $E$ verifying $\mu_0=\delta_0$ and
$\mu_{t+s}=(\mu_t \circ [S(s)]^{-1}) * \mu_s$ for any $s,t \geq 0$.
The semigroup \eqref{GMsgrp} is related to the stochastic differential equation
\begin{equation}\label{Mehler_SDE}
\begin{cases}
dZ(t) = AZ(t)dt + dY(t), & t>0 \\
Z(0)=x\in E
\end{cases}
\end{equation}
where $A:D(A)\subseteq E\to E$ is the infinitesimal generator of $(S(t))_{t\geq 0}$ and $Y(t)$ is a L\'{e}vy process
in $E$, i.e., a stochastic process with c\`{a}dl\`{a}g trajectories starting at $0$ and having
stationary and independent increments. For $\xi\in E$, $t>0$ we have
$\mathbb{E}[e^{i\langle\xi, Y(t)\rangle}]=\exp(-t\lambda(\xi))$ where $\lambda$ is a Sazonov continuous function related to $\mu_t$ through its characteristic function
$\hat{\mu}_t(\xi)=\exp(-\int_0^t\lambda((S(s))^*\xi)ds)$, see \cite{BRS96}, where $(S(s))^*$ denotes the dual semigroup of $S(s)$ and we identified $E$ with its dual $E^*$.
By the L\'{e}vy--Khinchine theorem, the function $\lambda$ is determined by its characteristics $[b,Q,\nu]$ with $b\in E$, $Q$ is a nonnegative definite symmetric trace-class operator on $E$
and $\nu$ is a L\'{e}vy measure, see Section \ref{section_Mehler} below. The semigroup $(P(t))_{t\geq 0}$ is related to
\eqref{Mehler_SDE} by
\[
[P(t)f](x)=\mathbb{E}[f(Z(t,x))],\qquad t\geq 0,\ x\in E, \, f \in B_b(E), 
\]
where $Z(t,x)$ is the (mild or weak) solution to \eqref{Mehler_SDE}.

In the special case when $Y(t)=Q^{1/2}W(t)$ with $W(t)$ a Brownian motion (i.e., $b \equiv 0$, $\nu\equiv 0$), then, for $t\in[0,\infty]$, setting
$Q_tx=\int_0^tS(s)Q(S(s))^*xds$, $x \in E$ and $\mu_t=\mathcal{N}(0,Q_t)$, the semigroup in \eqref{GMsgrp}   denotes the
Ornstein--Uhlenbeck semigroup $(R(t))_{t \ge 0}$ given by the classical Mehler formula. In this case the trajectories are
continuous, whereas in the general case $Y(t)$ may have jumps giving rise to nonlocal effects. Indeed,
the (weak) generator of the semigroup $(P(t))_{t\geq0}$ is in general a pseudo-differential operator acting on smooth functions as follows
\begin{equation}\label{generator_intro}
(\mathcal{L}f)(x)=\frac{1}{2} {\rm Tr}[QD^2f](x) + \langle Ax+b, D f(x)\rangle
+\int_E[f(x+y)-f(x)-\gen{Df(x),y}\chi_{B_1}(y)]d\nu(y).
\end{equation}
  
We refer to \cite{PZ07} for a general
introduction to these topics, to
\cite{App07,App15,BRS96,Chojnowska87,FR00,LR02,LR04,PT16,PZ06,Wang05}
as more specific basic references to generalized Mehler semigroups and to the very recent \cite{LR21} and the
reference therein for an updated account on the regularity theory, which we do not discuss here.

It is well known that in finite dimension and even for local operators as in \eqref{generator_intro} ($\nu\equiv 0$), the usual $L^p$
spaces with respect to the Lebesgue measure are not a natural setting for studying  elliptic and parabolic problems associated to $\mathcal{L}$. Much better settings are $L^p$
spaces with respect to invariant measure, i.e. a Borel probability measure $\mu$ on $E$ such that
\begin{align*}
\int_E P(t)f d\mu=\int_E f d\mu,\qquad\;\, t \ge 0,\, f \in B_b(E),
\end{align*}
or equivalently $\mu=(\mu\circ (S(t))^{-1})*\mu_t$ for every $t >0$, see formula \eqref{GMsgrp}.
Indeed, if such a measure exists, the semigroup $(P(t))_{t \ge 0}$ extends to a strongly continuous contractive semigroup (still denoted by $(P(t))_{t \ge 0}$ on $L^p(E, \mu)$ for every $p \in [1,+\infty)$.

In \cite{Nel73} it is proved that, under suitable conditions, the Ornstein--Uhlenbeck semigroup $(R(t))_{t \ge 0}$ defined as
\begin{align}\label{OU_Classical}
    [R(t)f](x)=\int_Ef(e^{-t}x+\sqrt{1-e^{-2t}}y)d\mu(x),\qquad t>0,\ f\in B_b(E)
\end{align}
where $\mu$ is the centred Gaussian measure with covariance operator $Q$, admits a unique invariant measure which is the Gaussian measure $\mathcal{N}(0,Q_\infty)=:\gamma$ and that $(R(t))_{t \ge 0}$ is hypercontractive in the $L^p(E,\gamma)$-scale. This means that $R(t)$ maps $L^p(E,\gamma)$ into $L^{q}(E,\gamma)$ for any $p,q \in (1,\infty)$ such that $q \in [p,1+(p-1)\|S(t)_{|_H}\|_{\mathcal{L}(H)}^{-2}]$ and any $t>0$, where $H$ is the Cameron-Martin space associated to $E$, and, further 
\begin{equation}\label{intro_hyper} \|R(t)f\|_{L^{q}(E, \gamma))}\le \|f\|_{L^p(E, \gamma)}.                         \end{equation}

Hypercontractivity estimates of this kind, and similar, are systematically studied by various authors (see, for example, \cite{BakryEmery,BGL14,Feissner1975,Gross75,Hariya18,NPY20,RW03}). Estimate \eqref{intro_hyper} is known to be equivalent to logarithmic Sobolev
inequalities 
\begin{equation}\label{log-sob-intro}
\int_E |f|^p\ln|f|d\gamma-\left(\int_E|f|^pd\gamma\right) \ln\left(\int_E|f|^p d\gamma\right)\le c_p\int_E|f|^{p-2}\|\nabla_H f\|^2d\gamma
\end{equation}
which hold true for any $p\geq 1$, $f$ regular enough with positive infimum and some positive constant $c_p$. Here $\nabla_H$ denotes the gradient along the Cameron--Martin space (see \cite{BFFZ24,Bog98} and the reference therein for more details). 
Logarithmic Sobolev inequalities, such as \eqref{log-sob-intro}, represent the counterpart of the Sobolev embeddings which fail to hold in the Gaussian case. These estimates have a wide range of consequences in probability and analysis; they are related to spectral gap and Poincaré-type inequalities and yield concentration of measure and transport/entropy bounds. A more in depth study of hypercontractivity, its relationship with logarithmic Sobolev inequalities and their consequences can be found for instance in
\cite{Ane, Ledoux01} and the reference therein.
Finally, estimates like \eqref{intro_hyper} and \eqref{log-sob-intro} have been generalized in various settings both in finite and infinite dimension (see \cite{AAL19, ADF25, AFP23, ALL13, BakryEmery, BF22, BobLed, BT06, Chafai04, CMG96, Hariya18, NPY20, OR16,RW03, Wang14}).

On the other hand, it is well known that they do not hold in general for generalized Mehler semigroups.
Indeed, in \cite{RW03} (see also Example \ref{example_1_Wang}) it is proved that the classical hypercontractivity estimate fails when the measures $\mu_t$ are not purely Gaussian. More precisely, it may happen that
\[
\|P(t)\|_{\mathcal{L}(L^p(E,\mu),L^q(E,\mu))}=\infty,
\qquad
\]
for every $1<p<q<\infty$.
 This phenomenon contradicts the
 expectation of obtaining in the general case hypercontractivity as in the classical sense, reflecting that for models with jumps, the classical $L^p$–$L^q$ scale may not be the best suited for this type of problem.

Even for logarithmic Sobolev inequalities, it has been shown in \cite{BobLed} that for Poisson processes (purely jump processes) we cannot expect the validity of the inequality \eqref{log-sob-intro}. Actually a modified version of it can be deduced.
Similar results have been recently obtained in a more general case in \cite{AFP23} where the entropy (with respect to invariant measure)
of positive measurable functions $f$ is estimated by the integral of some relative increments of $f$ with respect
to the L\'{e}vy measure $\nu$, which is charged to take into account the nonlocal effects.
It is important to point out that in this setting, differently from the Gaussian one, no connection between the modified logarithmic Sobolev inequalities proved and some hypercontractivity type property has been established.

Thus, naturally, it raises the question of the summability improvement property by the action of a generalized Mehler semigroup in the non-purely Gaussian case and of the occurrence of some related logarithmic Sobolev inequalities. 

In this paper, we address and give some partial answer to this problem. First, the abstract framework we work in, is that of skew operators on  separable real Banach spaces among which generalized Mehler semigroups represent a particular case. If $E$ is a real Banach space, we say that $T:E\to E$ is a skew map with respect to a Radon probability measure $\mu$ if there exists a probability measure $\rho$ such that $(\mu \circ T^{-1}) * \rho = \mu$, see Definition \ref{4.1}. We can consider a linear operator $P_T : B_b(E) \to B_b(E)$ defined as
\begin{equation*}
P_T f(x) := \int_E f(T(x) + y) \, d\rho(y),\quad\;\; f\in B_b(E).
\end{equation*} 
Clearly, in this setting the measure $\mu$ plays the role of the invariant measure for generalized Mehler semigroups  for $P_T$ and, as in that case, $P_T$ extends uniquely to a linear contraction (still denoted by $P_T$) on $L^p(E, \mu)$ for every $p \in [1,+\infty)$. 

The main contributions of this paper are two hypercontractivity-type estimates for $P_T$, see Theorem \ref{thm_hyper}. Such estimates are based on the fact that the invariant measure $\mu$ is an appropriate infinitely divisible measure that can be written as the convolution $\gamma*\pi$ where $\gamma$ is a Gaussian measure and  $\pi$ is an infinitely divisible, non-Gaussian probability measure. This type of decomposition is always possible if $\mu$ is sufficiently regular (see \cite[Theorem 4.20]{Hey}) and allows to consider the mixed-norm spaces $\mathcal{X}_{p,q}(E;\gamma,\pi)$ (see Section \ref{sec_mixed}) that, when $p<q$ are such that
$$L^q(E, \gamma*\pi)\subseteq \mathcal{X}_{p,q}(E;\gamma,\pi)\subseteq L^p(E; \gamma*\pi).$$
One of the hypercontractivity type estimate proved, states that for $p\in (1,\infty)$ 
\begin{equation}\label{hy1intro}
\|P_T f\|_{L^q(E,\gamma*\pi)}\le\|f\|_{\mathcal{X}_{p,q}(E;\gamma,\pi)},
\end{equation}
for a range of exponents $q>p$, depending only on $p$ and $T$.
Estimate \eqref{hy1intro} reduces to the classical hypercontractivity of the Ornstein--Uhlenbeck semigroup and reaffirms continuity of $P_T$ in $L^q(E, \pi)$ if $\pi$, respectively $\gamma$, is the Dirac measure concentrated at the origin.

Starting from inequality \eqref{hy1intro} and from the modified version of weak hypercontractivity estimate (see \eqref{weak_hyp_convoluted}), we are able to prove two modified versions of the classical logarithmic Sobolev inequality, see Theorems \ref{logSobolev_Strana} and \ref{thm_logSobolev_KL}.
Both of these inequalities do not involve the entropy of a function with respect to $\mu$ but they are given in terms of the integral with respect to $\pi$ of the entropy with respect to $\gamma$ of any smooth enough functions.

\section{Notations and preliminaries}

If $E$ is a Banach space, we denote by $E^*$ its topological dual endowed with the dual norm. For any two Banach spaces $E_1$ and $E_2$, we denote by $\mathcal{L}(E_1,E_2)$ the space of bounded linear operators from $E_1$ to $E_2$, endowed with its natural norm. If $E_1=E_2=E$, we simply write $\mathcal{L}(E)$. The symbol $\mathrm{Id}_E$ denotes the identity operator on $E$. We denote by $\langle\cdot,\cdot\rangle_{E\times E^*}$ the standard action of $E^*$ on $E$, i.e., $\langle x, x^*\rangle_{E\times E^*} := x^*(x)$ for all $x\in E$ and $x^*\in E^*$. For additional notation we refer the reader to \cite{FHH+11}.

For $k = 0,1,2,\ldots$, we denote by $C^{k}_{b}(E)$ the space of all $k$-times Fréchet differentiable functions $\varphi : E \to \mathbb{R}$ whose derivatives are continuous and bounded, endowed with its natural norm. When $k=0$, $C^{0}_{b}(E)$ coincides with the space of bounded continuous functions $\varphi : E \to \mathbb{R}$ endowed with the supremum norm, and we simply write $C_{b}(E)$ instead of $C^{0}_{b}(E)$. If the subscript $b$ is omitted, $C^{k}(E)$ denotes the space of all $k$-times Fréchet differentiable functions $\varphi : E \to \mathbb{R}$ whose derivatives are continuous but not necessarily bounded.

We introduce a generalized notion of differentiability by considering Dini derivatives, which are widely used in nonsmooth analysis to capture one-sided extremal rates of variation. We restrict our attention to the right Dini derivatives. Similar definitions and results can be formulated for the left Dini derivatives, and we refer the reader to \cite{Bru94,Cla90} for a more in-depth study of their properties. For any function $f : \mathbb{R} \to \mathbb{R}$ and any $x \in \mathbb{R}$, we consider the two right Dini derivatives of $f$ at $x$, defined as:
\begin{align}\label{Dini}
(D^+ f)(x) &= \limsup_{h \to 0^+} \frac{f(x+h)-f(x)}{h}, &
(D_+ f)(x) &= \liminf_{h \to 0^+} \frac{f(x+h)-f(x)}{h}.
\end{align}
For any $x \in \mathbb{R}$, we have $(D^+ f)(x), (D_+ f)(x)\in \mathbb{R}\cup\{\pm\infty\}$. Clearly, if $f$ is differentiable at $x$, then $(D^+ f)(x)=(D_+ f)(x)=f'(x)$.
The upper right Dini derivative $D^+$ is positively homogeneous and subadditive; namely, $(D^+(\lambda f))(x) = \lambda (D^+ f)(x)$ and $(D^+(f+g))(x) \le (D^+ f)(x) + (D^+ g)(x)$ for all $\lambda \ge 0$ and $g:\mathbb{R}\to \mathbb{R}$.

In the following proposition, we collect the main properties of the right Dini derivatives that we will need. These relate to calculus rules involving smooth functions.

\begin{pro}\label{Proposition_Dini_derivatives}
Let $f \in C^1(\mathbb{R})$, $g \in C(\mathbb{R})$, and $x_0 \in \mathbb{R}$. The following properties hold:
\begin{enumerate}[(i)]
    \item If $f(x_0) \ge 0$, then $(D^+(fg))(x_0) = f(x_0)(D^+g)(x_0) + f'(x_0)g(x_0)$.
    \item If $f(x_0) < 0$, then $(D^+(fg))(x_0) = f(x_0)(D_+g)(x_0) + f'(x_0)g(x_0)$.
    \item If $f'(g(x_0)) \ge 0$, then $(D^+(f \circ g))(x_0) = f'(g(x_0))(D^+g)(x_0)$.
    \item If $f'(g(x_0)) < 0$, then $(D^+(f \circ g))(x_0) = f'(g(x_0))(D_+g)(x_0)$.
    \item If $f(x_0) \ge 0$ and $g(x_0)\neq 0$, then 
    \begin{align*}
        \left(D^+\frac{f}{g}\right)(x_0)=\frac{f'(x_0)g(x_0)-f(x_0)(D_+g)(x_0)}{(g(x_0))^2}.
    \end{align*}
    \item If $f(x_0) < 0$ and $g(x_0)\neq 0$, then 
    \begin{align*}
        \left(D^+\frac{f}{g}\right)(x_0)=\frac{f'(x_0)g(x_0)-f(x_0)(D^+g)(x_0)}{(g(x_0))^2}.
    \end{align*}
\end{enumerate}
Analogous results hold for $D_+$ with obvious modifications.
\end{pro}

\begin{proof}
The proofs of properties (i) and (ii) can be found in \cite[Appendix I]{RHL12}, while (v) and (vi) are immediate consequences of the preceding properties. Hence, it remains to prove (iii) and (iv).

Recall that for any sequence $(\xi_n)_{n\in\N}\subseteq\R$, we have $\limsup_{n\ra\infty} c \xi_n = c \limsup_{n\ra\infty} \xi_n$ if $c \ge 0$, and $\limsup_{n\ra\infty} c \xi_n = c \liminf_{n\ra\infty} \xi_n$ if $c < 0$.

By the mean value theorem, $f(y)-f(z) = f'(\theta)(y-z)$ for some $\theta$ between $y$ and $z$. Setting $y=g(x_0+h)$ and $z=g(x_0)$, we obtain
\[
\frac{f(g(x_0+h)) - f(g(x_0))}{h} = f'(\theta_h) \frac{g(x_0+h)-g(x_0)}{h},
\]
where $\theta_h$ lies between $g(x_0+h)$ and $g(x_0)$. Since $g$ is continuous, $\theta_h$ tends to $g(x_0)$ as $h \to 0^+$, which implies $f'(\theta_h) \to f'(g(x_0))$. The result follows by applying the aforementioned properties of the limit superior with $c = f'(g(x_0))$.
\end{proof}

Let $E$ be a Banach space endowed with its Borel $\sigma$-algebra $\mathcal{B}(E)$. Let $\nu$ be a finite Radon measure on $E$, that is, a finite Borel measure on $E$ such that for every $\varepsilon > 0$, there exists a compact set $K \subseteq E$ with $\nu(E \smallsetminus K) < \varepsilon$. Recall that if $E$ is separable, then every finite Borel measure on $E$ is Radon (see \cite[Proposition 1.1.3]{Lin86}). The characteristic function of $\nu$ is the map $\hat{\nu} : E^* \to \mathbb{C}$ defined by
\[
\hat{\nu}(x^*) = \int_E e^{ix^*(x)}  d\nu(x),\qquad x^*\in E^*.
\]
Given two finite Radon measures $\mu,\nu$ on $E$, their convolution $\mu * \nu$ is the Radon measure on $E$ defined by
\[
(\mu * \nu)(A) := \int_E \mu(A - y)\, d\nu(y),
\]
for $A \in \mathcal{B}(E)$, where $A - y = \{ x \in E \,|\, x+y \in A \}$. Equivalently, for every bounded Borel function $f : E \to \R$, one has
\[
\int_E f(z)d(\mu * \nu)(z) 
= \int_E \int_E f(x+y)d\mu(x)d\nu(y).
\]
If $T\in \mathcal{L}(E)$, $\mu\circ T^{-1}$ denotes the image measure $(\mu\circ T^{-1})(A):=\mu(T^{-1}(A))$ for $A \in \mathcal{B}(E)$.
We denote the set of bounded Borel functions $f:E\to\R$ with the symbol $B_b(E)$. If $\gamma$ is a Gaussian measure on $E$ with mean $a\in E$ and covariance operator $Q$ we will write $\gamma=\mathcal{N}(a,Q)$ (see \cite{Bog98} for informations about finite and infinite dimensional Gaussian measures). Due to their use in this paper, we need to recall the definition of an infinitely divisible measure. A finite measure $\mu$ on $E$ is called infinitely divisible if, for every $n\in\N$, there exists a finite measure $\mu_n$ on $E$ such that 
\begin{align*}
    \underbrace{\mu_n*\cdots*\mu_n}_{n \text{ times}} = \mu.
\end{align*}
We refer to \cite[Chapter 2.3]{Hey} for an in-depth discussion of this type of measure.

For any finite measure space $(\Omega, \mathcal{F}, \mu)$, the space $L^0(\Omega,\mu)$ of equivalence classes of $\mu$-measurable functions $f:\Omega\to\R$ (modulo the $\mathcal{F}$-almost everywhere equality) is a complete metric space when endowed with the metric
\begin{align*}
    d_\mu(f,g) = \int_\Omega \min\{ |f(z) - g(z)|, 1 \} \, d\mu(z),\qquad f,g\in L^0(\Omega,\mu),
\end{align*}
which metrizes the topology of the convergence in measure in $(\Omega,\mathcal{F},\mu)$ (see \cite[Theorem 245E]{Fre01}).
For notations and preliminaries about measures and measurable spaces we refer to to \cite{Bil95,Bog07,Fre01,Lin86}.


Recall that if $(\Omega,\mathcal{F},\mu)$ is a finite measure space, $p \in [1,\infty)$, and $Y$ is a Banach space, the Bochner space $L^p(\Omega,\mu;Y)$ is defined as the space of (equivalence classes of) strongly measurable functions $f : \Omega \to Y$ such that
\[\|f\|_{L^p(\Omega,\mu;Y)} := \left( \int_\Omega \|f(x)\|_Y^p  d\mu(x) \right)^{1/p} < \infty.\]
For further details, see \cite{DU77}. If, in particular, $Y = L^q(\Sigma, \sigma)$ for some $q \in [1, \infty)$ and some finite measure space $(\Sigma, \mathcal{G}, \sigma)$, then $L^p(\Omega, \mu; L^q(\Sigma, \sigma))$ denotes the mixed-norm space $L^{(p,q)}((\Omega,\mu),(\Sigma,\sigma))$, in the sense of \cite{BP61}, that is the set of functions $F: \Omega \times \Sigma \to \mathbb{R}$ such that
\begin{align*}
    \| F \|_{L^p(\Omega, \mu; L^q(\Sigma, \sigma))} 
    = \left( \int_{\Omega} \left( \int_{\Sigma} |F(x_1, x_2)|^q d\sigma(x_2) \right)^{\frac{p}{q}} d\mu(x_1) \right)^{\frac{1}{p}} < \infty.
\end{align*}

\subsection{Tensor products}\label{sect_Tensor}
We recall the definition of the Chevet--Saphar norm on tensor products (see \cite[Chapter 7]{DF93}, \cite[Chapter 6]{Ryan02} or \cite[Chapter 1]{LC85} where it is called $p$-nuclear norm). To do so we need a preliminary definition: if $Y$ is a Banach space, $y_1,\ldots,y_n\in Y$ and $r\in(1,\infty]$ we define
\begin{align*}
    \mu_{r,Y}(y_1,\ldots,y_n)=\sup\left\{\norm{\sum_{i=1}^n\lambda_iy_i}_Y\,\middle|\ \left(\sum_{i=1}^n|\lambda_i|^{r'}\right)^{\frac{1}{r'}}\leq 1\right\},
\end{align*}
here $r$ and $r'$ are conjugate exponents, i.e. $r^{-1}+(r')^{-1}=1$ with the standard convention that $\infty'=1$. In the case $r=1$ we let
\begin{align*}
    \mu_{1,Y}(y_1,\ldots,y_n)=\sup\left\{\norm{\sum_{i=1}^n\lambda_iy_i}_Y\,\middle|\ \max_{i=1,\ldots,n}|\lambda_i|\leq 1\right\}.
\end{align*}
Now let $X,Y$ be two Banach spaces and $z\in X\otimes Y$. For any $r\in(1,\infty)$ the Chevet--Saphar norm of $z$ is defined by
\begin{align*}
    \alpha^{X,Y}_r(z)=\inf\left\{\left(\sum_{i=1}^n\|x_i\|_X^r\right)^{\frac{1}{r}}\mu_{r',Y}(y_1,\ldots,y_n)\,\middle|\, z=\sum_{i=1}^nx_i\otimes y_i\right\},
\end{align*}
where the 
infimum is taken with respect all the possible representations of $z$. 
With the symbol $X\otimes_{r}Y$ we denote the closure of $X\otimes Y$ with respect to the crossnorm $\alpha_p^{X,Y}$.

In the following proposition we collect some known results whose proofs can be found in \cite[Lemmas 1.47, 1.49 and 1.51, Theorem 1.50 and Corollary 1.52]{LC85}.

\begin{pro}\label{prop1}
Let $Y$ be a Banach space and $(\Omega,\mathcal{F},\mu)$ be a finite measure space. 
\begin{enumerate}[(i)]
    \item Let $\mathscr{S}(\Omega,\mathcal{F};Y)$ be the set of $Y$-valued $\mathcal{F}$-measurable step functions with domain $\Omega$ and let $q\in[1,\infty)$. The embedding $\iota_1: \mathscr{S}(\Omega,\mathcal{F};Y) \longrightarrow Y \otimes_q L^q(\Omega,\mu)$ defined as
    \[ 
    \iota_1\left(\sum_{i=1}^n \chi_{A_i} y_i\right) = \sum_{i=1}^n y_i \otimes \chi_{A_i},
    \]  
    where $n\in\N$, $A_i \in \mathcal{F}$ and $y_i \in Y$, for any $i=1,\ldots,n$, extends by density to $L^q(\Omega,\mu;Y)$. Furthermore $\|\iota_1\|_{\mathcal{L}(L^q(\Omega,\mu;Y), Y \otimes_q L^q(\Omega,\mu))} = 1$.

    \item For $q\in[1,\infty]$, the embedding $\iota_2: L^q(\Omega,\mu) \otimes Y \longrightarrow L^q(\Omega,\mu;Y)$ defined as
    \[
    \iota_2\left(\sum_{i=1}^n f_i \otimes y_i\right)(s) = \sum_{i=1}^n f_i(s) y_i,
    \]  
    for $n\in\N$, $s\in \Omega$, $f_i \in L^q(\Omega,\mu)$ and $y_i \in Y$, for $i=1,\ldots,n$, extends by density to $L^q(\Omega,\mu) \otimes_q Y$. Furthermore $\|\iota_2\|_{\mathcal{L}(L^q(\Omega,\mu) \otimes_q Y,L^q(\Omega,\mu;Y))} = 1$.
\end{enumerate}
Moreover, if $q\in[1,\infty)$ and $Y = L^q(\Sigma,\sigma)$ for a finite measure space $(\Sigma,\mathcal{G},\sigma)$, then all the following spaces are isometrically isomorphic
\begin{align*}
    L^q(\Sigma,\sigma)\otimes_q L^q(\Omega,\mu)  = L^q(\Omega,\mu)\otimes_q L^q(\Sigma,\sigma) &= L^q(\Omega,\mu; L^q(\Sigma,\sigma))\\
    &= L^q(\Sigma,\sigma; L^q(\Omega,\mu))=L^q(\Omega\times \Sigma,\mu\times\sigma).
\end{align*}
\end{pro}
\begin{rmk}\label{dueo}\textrm{Let $(\Omega_i,\mathcal{F}_i,\mu_i)$, $i=1,2$ be two finite measure spaces and let $p,q \in [1,\infty)$. Thanks to Proposition \ref{prop1}(i), we can identify $L^q(\Omega_1,\mu_1;L^p(\Omega_2,\mu_2))$ with a subspace of $L^p(\Omega_2,\mu_2)\otimes_q L^q(\Omega_1,\mu_1)$. Analogously, by Proposition \ref{prop1}(ii), we can identify $L^q(\Omega_1,\mu_1) \otimes_q L^p(\Omega_2,\mu_2)$ with a subspace of $L^q(\Omega_2,\mu_2;L^p(\Omega_1, \mu_1))$.}
\end{rmk}

\section{The mixed-norm spaces $\mathcal{X}_{p,q}(E;\mu_1, \mu_2)$}\label{sec_mixed}
We now introduce a set of functions that will play a central role in our analysis. 
Let $p, q \in [1, \infty)$ and let $\mu_1,\mu_2$ be  two probability measures on a measurable space $E$. We define the space
\begin{align}
    X^{p,q}(E;\mu_1,\mu_2) := \big\{ F \in L^{q}(E,\mu_2;L^p(E,\mu_1)) \,\big|\, 
    \exists f \in L^0(E,\mu_1 * \mu_2)\text{ s.t. } F(x,y) = f(x+y)\big\}\label{mixed_Xpq}.
\end{align}
The equality in \eqref{mixed_Xpq} has to be understood to hold for $(\mu_1 \times \mu_2)$-a.e. $(x,y)\in E\times E$.

First of all, let us observe that $X^{p,q}(E;\mu_1,\mu_2)$ is well-defined. Indeed, let $f, g: E \to \mathbb{R}$ be two measurable functions such that $f = g$ $(\mu_1 * \mu_2)$-almost everywhere in $E$, and define $F(x,y) := f(x+y)$ and $G(x,y) := g(x+y)$. Then
\begin{align*}
    d_{\mu_1 \times \mu_2}(F, G)
    &= \int_E \int_E \min\{ |F(x,y) - G(x,y)|, 1 \}  d\mu_1(x) \, d\mu_2(y) \notag \\
    &= \int_E \int_E \min\{ |f(x+y) - g(x+y)|, 1 \} d\mu_1(x) \, d\mu_2(y) \notag \\
    &= \int_E \min\{ |f(z) - g(z)|, 1 \} d(\mu_1 * \mu_2)(z) = d_{\mu_1 * \mu_2}(f, g) = 0,
\end{align*}
so that $F = G$ $(\mu_1 \times \mu_2)$-almost everywhere in $E\times E$. This fact shows that each $F \in X^{p,q}(E;\mu_1,\mu_2)$ admits a unique  $f \in L^0(E, \mu_1 * \mu_2)$ such that $F(x,y) = f(x+y)$ $(\mu_1 \times \mu_2)$-a.e. $(x,y)\in E\times E$.

This fact allows us to define a linear injective map
\[
\tau: X^{p,q}(E;\mu_1,\mu_2) \ra L^0(E, \mu_1 * \mu_2),
\]
associating to each $F \in X^{p,q}(E;\mu_1,\mu_2)$ the unique $f = \tau(F) \in L^0(E, \mu_1 * \mu_2)$ satisfying $F(x,y) = f(x+y)$ for $(\mu_1 \times \mu_2)$-a.e. $(x,y) \in E \times E$. We set
\[
\mathcal{X}_{p,q}(E;\mu_1,\mu_2): = \tau(X^{p,q}(E;\mu_1,\mu_2))
\]
equipped with the norm inherited from the mixed-norm space. More explicitly, for $f \in \mathcal{X}_{p,q}(E;\mu_1,\mu_2)$, we define
\[
\|f\|_{\mathcal{X}_{p,q}(E;\mu_1,\mu_2)} = \left( \int_E \left( \int_E |f(x+y)|^p \, d\mu_1(x) \right)^{\frac{q}{p}} d\mu_2(y) \right)^{\frac{1}{q}}.
\]
In the sequel, we shall simply write $\norm{\cdot}_{p,q}$ instead of $\norm{\cdot }_{\mathcal{X}_{p,q}(E;\mu_1,\mu_2)}$ whenever no ambiguity arises. Observe that with this conventions it holds that $\tau$ is a surjective isometry from $X^{p,q}(E;\mu_1,\mu_2)$ to $\mathcal{X}_{p,q}(E;\mu_1,\mu_2)$.

\begin{pro}\label{pro_spacesXpq}
For any $p,q\in[1,\infty)$, the space \( (\mathcal{X}_{p,q}(E;\mu_1,\mu_2), \norm{\cdot}_{p,q}) \) is a Banach space. Furthermore the space $B_b(E)$ is a dense subspace of $\mathcal{X}_{p,q}(E;\mu_1,\mu_2)$ and
\begin{equation} \label{inclusion}
    L^{\max\{p,q\}}(E, \mu_1 * \mu_2) \subseteq \mathcal{X}_{p,q}(E;\mu_1,\mu_2) \subseteq L^{\min\{p,q\}}(E, \mu_1 * \mu_2),
\end{equation}
with continuous embeddings. 
\end{pro}

\begin{proof}
Since the map $\tau$ is a surjective isometry from $X^{p,q}(E;\mu_1,\mu_2)$ onto $\mathcal{X}_{p,q}(E;\mu_1,\mu_2)$, it follows that $\mathcal{X}_{p,q}(E;\mu_1,\mu_2)$ is a Banach space, as a consequence of the completeness of $X^{p,q}(E;\mu_1,\mu_2)$. Indeed let $(F_n)_{n\in\N}\subseteq X^{p,q}(E;\mu_1,\mu_2)$ be a Cauchy sequence, by the completeness of $L^{q}(E,\mu_2;L^p(E,\mu_1))$ it admits a limit $F$ belonging $L^{q}(E,\mu_2;L^p(E,\mu_1))$. We claim that the sequence $f_n=\tau(F_n)$ converges in probability to a $f\in L^0(E,\mu_1*\mu_2)$ such that $\tau F=f$. By the completeness of the metric $(L^0(E,\mu_1*\mu_2),d_{\mu_1*\mu_2})$ (see \cite[Theorem 245E]{Fre01}) it is enough to prove that $(f_n)_{n\in\N}$ is a Cauchy sequence. Observe that by the H\"older inequality 
\begin{align*}
    d_{\mu_1*\mu_2}(f_n,f_m) &=\int_E\min\{|f_n(z)-f_m(z)|,1\}d(\mu_1*\mu_2)(z)\\
    &= \int_E\int_E\min\{|f_n(x+y)-f_m(x+y)|,1\}d\mu_1(x)d\mu_2(y)\\
    &= \int_E\int_E\min\{|F_n(x,y)-F_m(x,y)|,1\}d\mu_1(x)d\mu_2(y)\\
    &\leq \int_E\int_E|F_n(x,y)-F_m(x,y)|d\mu_1(x)d\mu_2(y)\\
    &\leq \norm{F_n-F_m}_{L^{q}(E,\mu_2;L^p(E,\mu_1))}.
\end{align*}
So by the fact that $(F_n)_{n\in\N}$ is a Cauchy sequence in $L^{q}(E,\mu_2;L^p(E,\mu_1))$ it holds that there exists $f\in L^0(E,\mu_1*\mu_2)$ that is the limit in measure of the sequence $(f_n)_{n\in\N}\subseteq L^0(E,\mu_1*\mu_2)$. Now, setting $G(x,y):=f(x+y)$ for $(\mu_1\times \mu_2)$-almost every $(x,y)\in E\times E$. Observing that, by the H\"older inequality, $(F_n)$ converges to $F$ with respect to the $d_{\mu_1\times\mu_2}$ distance as $n \to +\infty$, and that by the continuity of the $d_{\mu_1\times \mu_2}(\cdot,\cdot)$ we get that
\begin{align*}
|d_{\mu_1\times \mu_2}(F_n,G)-d_{\mu_1\times \mu_2}(F,G)|\le d_{\mu_1\times \mu_2}(F_n,F), \qquad\;\, n \in \N
\end{align*}
it follows that $d_{\mu_1\times\mu_2}(F,G) =\lim_{n\ra\infty}d_{\mu_1\times\mu_2}(F_n,G)$.
Thus,
\begin{align*}
    d_{\mu_1\times\mu_2}(F,G) &=\lim_{n\ra\infty}\int_E\int_E\min\{|F_n(x,y)-G(x,y)|,1\}d\mu_1(x)d\mu_2(y)\\
    &=\lim_{n\ra\infty}\int_E\int_E\min\{|f_n(x+y)-f(x+y)|,1\}d\mu_1(x)d\mu_2(y)\\
    &=\lim_{n\ra\infty}\int_E\min\{|f_n(z)-f(z)|,1\}d(\mu_1*\mu_2)(z)= \lim_{n\ra\infty}d_{\mu_1*\mu_2}(f_n,f)=0.
\end{align*}
So $F(x,y)=G(x,y)=f(x+y)$ for $(\mu_1\times\mu_2)$-almost every $(x,y)\in E\times E$. This concludes the proof of the completeness of $\mathcal{X}_{p,q}(E;\mu_1,\mu_2)$.

To see that $B_b(E)$ is dense in $\mathcal{X}_{p,q}(E;\mu_1,\mu_2)$, let $f$ be a Borel version of an element of $\mathcal{X}_{p,q}(E;\mu_1,\mu_2)$. Then, for every $n\in\N$ the function
\begin{align*}
    f_n(z)=f(z)\chi_{\{|f|\leq n\}}(z),\qquad z\in E.
\end{align*}
belongs to $B_b(E)$. Moreover, being $|f-f_n|\leq |f|$ and $f(x+y)-f_n(x+y)\ra 0$ for $(\mu_1\times\mu_2)$-a.e. $(x,y)\in E\times E$, by the dominated convergence theorem it holds that $f_n$ converges to $f$ in $\mathcal{X}_{p,q}(E;\mu_1,\mu_2)$.

We prove \eqref{inclusion} only in the case $q>p$, since the case $q<p$ follows similarly and the case $q=p$ is obvious. The Jensen inequality implies that $L^q(E, \mu_1 * \mu_2)$ is continuously embedded in $\mathcal{X}_{p,q}(E;\mu_1,\mu_2)$. Indeed, since \( q > p \), we have
\begin{align*}
    \|f\|_{p,q} &= \left( \int_E \left( \int_E |f(x + y)|^p \, d\mu_1(x) \right)^{\frac{q}{p}} d\mu_2(y) \right)^{\frac{1}{q}} \\
    &\leq \left( \int_E \int_E |f(x + y)|^q \, d\mu_1(x) d\mu_2(y) \right)^{\frac{1}{q}} = \|f\|_{L^q(E, \mu_1 * \mu_2)}.
\end{align*}
Again, using Jensen's inequality and the assumption \( q > p \), we get
\begin{align*}
    \|f\|_{p,q} &= \left( \int_E \left( \int_E |f(x + y)|^p \, d\mu_1(x) \right)^{\frac{q}{p}} d\mu_2(y) \right)^{\frac{1}{q}} \\
    &\geq \left( \int_E \int_E |f(x + y)|^p \, d\mu_1(x) d\mu_2(y) \right)^{\frac{1}{p}} = \|f\|_{L^p(E, \mu_1 * \mu_2)}.
\end{align*}
This shows that \( \mathcal{X}_{p,q}(E;\mu_1,\mu_2)\) is continuously embedded in \(L^p(E, \mu_1 * \mu_2) \).
\end{proof}

We collect a few elementary properties of the spaces 
\( \mathcal{X}_{p,q}(E;\mu_1,\mu_2) \) in the next remark.

\begin{rmk}
The following facts are immediate consequences of the definition. Here $p,q$ belong to $[1,\infty)$.
\begin{enumerate}[(i)]
    \item \( \mathcal{X}_{p,p}(E;\mu_1,\mu_2)=L^p(E,\mu_1*\mu_2) \).
    
    \item If \( \mu \) is a probability measure on $E$, then 
    \( \mathcal{X}_{p,q}(E;\delta_0,\mu)=L^q(E,\mu) \) and \( \mathcal{X}_{p,q}(E;\mu,\delta_0)=L^p(E,\mu) \), where $\delta_0$ denotes the Dirac measure concentrated at the origin.

\end{enumerate}
\end{rmk}

In the sequel, it will be useful to consider the inverse of the map $\tau$ as done in \cite{AvN15} just in the case $L^2(E,\mu_1*\mu_2)$ (i.e. $p=q=2$). 

\begin{pro}\label{pro_Fbullet}
   For $f\in B_b(E)$ we consider $F_{f}\in B_b(E\times E)$ defined as
    \begin{align*}
        F_f(x,y)=f(x+y),\qquad x,y\in E.
    \end{align*}
    For every $p,q\in[1,\infty)$, the map $F_\bullet:B_b(E)\ra B_b(E\times E)$ defined as $F_\bullet(f)=F_f$ extends uniquely to an isometry from $\mathcal{X}_{p,q}(E;\mu_1,\mu_2)$ to $X^{p,q}(E;\mu_1,\mu_2)$. Moreover $F_f(x,y)=f(x+y)$ for $(\mu_1\times\mu_2)$-a.e. $(x,y)\in E\times E$ for any version of $f\in \mathcal{X}_{p,q}(E;\mu_1,\mu_2)$ and 
    \begin{align}\label{Koko}
        \tau\circ F_\bullet=\mathrm{Id}_{\mathcal{X}_{p,q}(E;\mu_1,\mu_2)},\qquad \qquad F_\bullet\circ\tau=\mathrm{Id}_{X^{p,q}(E;\mu_1,\mu_2)}.
    \end{align}
\end{pro}

\begin{proof}
    The fact that $F_\bullet$ is well defined is a consequence of the continuity of the sum. Now let $f\in \mathcal{X}_{p,q}(E;\mu_1,\mu_2)$ and let $(f_n)_{n\in\N}\subseteq B_b(E)$ converging to $f$ in $\mathcal{X}_{p,q}(E;\mu_1,\mu_2)$ (such a sequence exists by Proposition \ref{pro_spacesXpq}). Observe that for every $n,m\in\N$ it holds
    \begin{align*}
        \|F_{f_n}-F_{f_m}\|^q_{L^q(E,\mu_2;L^p(E,\mu_1))} &=\int_E\left(\int_E|F_{f_n}(x,y)-F_{f_m}(x,y)|^p d\mu_1(x)\right)^{\frac{q}{p}}d\mu_2(y)\\
        &=\int_E\left(\int_E|f_n(x+y)-f_m(x+y)|^p d\mu_1(x)\right)^{\frac{q}{p}}d\mu_2(y)\\
        &=\|f_n-f_m\|^q_{p,q}.
    \end{align*}
    So, since $X^{p,q}(E;\mu_1,\mu_2)$ is complete (proof of Proposition \ref{pro_spacesXpq}), the sequence $(F_{f_n})_{n\in\N}$ converges to an element $G\in X^{p,q}(E;\mu_1,\mu_2)$. Then, set $F_f(x,y)=f(x+y)$, observe that 
    \begin{align*}
        \|F_f-G\|_{L^q(E,\mu_2;L^p(E,\mu_1))}&\leq \|F_f-F_{f_n}\|_{L^q(E,\mu_2;L^p(E,\mu_1))}+\|F_{f_n}-G\|_{L^q(E,\mu_2;L^p(E,\mu_1))}\\
        &= \|f-f_n\|_{p,q}+\|F_{f_n}-G\|_{L^q(E,\mu_2;L^p(E,\mu_1))}.
    \end{align*}
Thus, since the right-hand side of the previous chain of inequalities vanishes as $n \to +\infty$, we get $G(x,y)=F_f(x,y)=f(x+y)$ for $(\mu_1\times\mu_2)$-a.e. $(x,y)\in E\times E$. Finally, the facts that $F_\bullet$ is an isometry and \eqref{Koko} follow by standard arguments.
\end{proof}

\begin{rmk}\textrm{As immediate consequence of Proposition \ref{pro_Fbullet} we get that $f \in \mathcal{X}_{p,q}(E;\mu_1, \mu_2)$ if and only if $F_f \in L^q(E, \mu_2;L^p(E, \mu_1))$ and $\|f\|_{p,q}=\|F_f\|_{L^q(E, \mu_2;L^p(E, \mu_1))}$.
In particular for any $r\in [1,\infty)$, by Proposition \ref{prop1}, it holds 
\begin{align}\label{formula_in_proof}
\|g\|_{L^r(E,\mu_1*\mu_2)}=\|F_g\|_{L^r(E\times E,\mu_1\times\mu_2)}.    
\end{align}}
\end{rmk}

In view of the results we are interested in, it is relevant to note that both inclusions in \eqref{inclusion} can be strict. To show this fact, we perform two one-dimensional examples. However, we point out that analogous constructions can be carried out in $\mathbb{R}^d$ for any $d > 1$.
In the first example,  we consider two probability measures \( \mu_1 \) and \( \mu_2 \) on \( \mathbb{R} \), which are also of interest from an applied perspective. Indeed, in \cite[Example 3.2]{RW03}, the authors show that the generalized Mehler semigroup, whose invariant measure is \( \mu_1 * \mu_2 \), is not hypercontractive in the classical sense. In Section \ref{sect_MainResults}, we will show that  a form of hypercontractivity still holds true in this case (see Theorem \ref{thm_hyper}).

\begin{example}\label{ex_Wang1}
 Let \( E = \mathbb{R} \), and consider the probability measures
\[
d\mu_1(x) = \frac{1}{\sqrt{2\pi}} e^{-\frac{x^2}{2}} \, dx, \qquad d\mu_2(y) = \frac{c_\alpha}{1 + |y|^\alpha} \, dy,
\]
where \( \alpha \in (1,3) \) and \( c_\alpha \) is a suitable normalization constant. To fix the ideas, let assume that $q>p$.

First, we show that the function \( f(\xi) = |\xi|^{-\beta} \), with \( \frac{1}{q} \le \beta < \frac{1}{p} \), belongs to \( \mathcal{X}_{p,q}(\R;\mu_1,\mu_2) \) but not to \( L^q(\mathbb{R}, \mu_1 * \mu_2) \). Indeed,
\begin{align*}
\|f\|_{L^q(\mathbb{R}, \mu_1 * \mu_2)}^q 
&= \int_{\mathbb{R}} \int_{\mathbb{R}} |f(x+y)|^q d\mu_1(x) d\mu_2(y) \\
&\ge \frac{c_\alpha}{\sqrt{2\pi}} \int_{-1}^{1} \int_{-1}^{1} \frac{1}{|x+y|^{\beta q}} e^{-\frac{x^2}{2}} \frac{1}{1 + |y|^\alpha}dxdy \\
&\ge \frac{c_\alpha}{2\sqrt{2e\pi}} \int_{-1}^{1} \int_{-1}^{1} \frac{1}{|x+y|^{\beta q}} dxdy,
\end{align*}
and, by our choice of $\beta$, the last integral diverges. On the other hand,
\begin{align*}
\|f\|_{p,q}^q &=\int_\mathbb{R}\left(\int_\mathbb{R}|f(x+y)|^pd\mu_1(x)\right)^{\frac{q}{p}}d\mu_2(y) \\
&= \int_{\mathbb{R}} \left( \int_{\{x \in \mathbb{R} \,|\, |x+y| \ge 1\}} \frac{1}{|x+y|^{\beta p}} d\mu_1(x) \right)^{\frac{q}{p}} d\mu_2(y) \\
&\quad + \int_{\mathbb{R}} \left( \int_{\{x \in \mathbb{R} \,|\, |x+y| < 1\}} \frac{1}{|x+y|^{\beta p}} d\mu_1(x) \right)^{\frac{q}{p}} d\mu_2(y) \\
&\le 1 + \int_{\mathbb{R}} \left( \int_{-1 - y}^{1 - y} \frac{1}{|x + y|^{\beta p}} dx \right)^{\frac{q}{p}} d\mu_2(y) \\
&= 1 + \int_{\mathbb{R}} \left( \int_{-1}^{1} \frac{1}{|t|^{\beta p}}  dt \right)^{\frac{q}{p}}  d\mu_2(y) < \infty,
\end{align*}
due to the choice of \( \beta \).

Next, we show that the function \( g(\xi) = |\xi|^\omega \), with \( \frac{\alpha - 1}{q} \le \omega < \frac{\alpha - 1}{p} \), belongs to \( L^p(\mathbb{R}, \mu_1 * \mu_2) \) but not to \( \mathcal{X}_{p,q}(\R;\mu_1,\mu_2) \). Indeed,
\begin{align}
\|g\|_{p,q}^q &= \frac{c_\alpha}{(2\pi)^{\frac{q}{2p}}} \int_{\mathbb{R}} \left( \int_\mathbb{R} |x+y|^{\omega p} e^{-\frac{x^2}{2}} \, dx \right)^{\frac{q}{p}} \frac{1}{1+|y|^\alpha} \, dy \notag \\
&\ge \frac{c_\alpha}{(2e\pi)^{\frac{q}{2p}}} \int_{-\infty}^{-2} \left( \int_0^1 |x+y|^{\omega p} \, dx \right)^{\frac{q}{p}} \frac{1}{1+|y|^\alpha} \, dy \notag \\
&= \frac{c_\alpha}{(2e\pi(\omega p+1)^2)^{\frac{q}{2p}}} \int_{1}^{\infty} \left( (1+s)^{\omega p+1} - s^{\omega p+1} \right)^{\frac{q}{p}} \frac{1}{1 + (1+s)^\alpha} \, ds. \label{example_conti}
\end{align}
Let
\(
r(s) := \left( (1+s)^{\omega p+1} - s^{\omega p+1} \right)^{\frac{q}{p}}\frac{1}{1 + (1+s)^\alpha},
\)
and observe that for large \( s \),
\begin{align} \label{r_asymp}
r(s) &= s^{(\omega p + 1)\frac{q}{p}} \left( \left( 1 + \frac{1}{s} \right)^{\omega p + 1} - 1 \right)^{\frac{q}{p}} \frac{1}{1 + (1+s)^\alpha} 
\sim \frac{(\omega p + 1)^{\frac{q}{p}}}{s^{\alpha - \omega q}}.
\end{align}
The asymptotic behaviour in \eqref{r_asymp}, combined with our assumption on \( \omega \), implies that the integral in \eqref{example_conti} diverges.
On the other hand,
\begin{align}
\|g\|_{L^p(\mathbb{R}, \mu_1 * \mu_2)}^p 
&= \frac{c_\alpha}{\sqrt{2\pi}} \int_{\mathbb{R}} \int_{\mathbb{R}} |x + y|^{\omega p} \frac{e^{-x^2/2}}{1 + |y|^\alpha} \, dx \, dy \notag \\
&\le \frac{c_\alpha \max\{1, 2^{\omega p-1}\}}{\sqrt{2\pi}} \int_{\mathbb{R}} \int_{\mathbb{R}} \left( |x|^{\omega p} + |y|^{\omega p} \right) \frac{e^{-x^2/2}}{1 + |y|^\alpha} \, dx \, dy \notag \\
&= \frac{\max\{1, 2^{\omega p-1}\}}{\sqrt{2\pi}} \int_{\mathbb{R}} |x|^{\omega p} e^{-x^2/2} \, dx 
+ c_\alpha \max\{1, 2^{\omega p-1}\} \int_{\mathbb{R}} \frac{|y|^{\omega p}}{1 + |y|^\alpha} \, dy. \label{example_conti2}
\end{align}
The first integral in \eqref{example_conti2} is finite since \( \omega > 0 \), and the second is also finite due to the condition \( \alpha - \omega p > 1 \), which holds by our choice of \( \omega \).

Therefore, both inclusions in \eqref{inclusion} are strict, i.e.,
\[
L^q(\mathbb{R}, \mu_1 * \mu_2) \subsetneq \mathcal{X}_{p,q}(\mathbb{R};\mu_1,\mu_2) \subsetneq L^p(\mathbb{R}, \mu_1 * \mu_2).
\]
\end{example}

The next example is taken from \cite[Example 3.3]{RW03}. We verify that also in this case, both inclusions in \eqref{inclusion} are strict. However, in this case, the authors in \cite{RW03} establish that a hypercontractivity-type estimate holds true for the associated generalised Mehler semigroup.

\begin{example}
   Let \( E = \mathbb{R} \) and consider the probability measures
    \[
    d\mu_1(x) = \frac{1}{\sqrt{2\pi}} e^{-\frac{x^2}{2}} \, dx, \qquad d\mu_2(y) = \frac{1}{2}e^{-|y|} \, dy.
    \]
As in the previous example, let assume that $q>p$, fix $\frac{1}{q}\leq \alpha<\frac{1}{p}$ and consider the function $f(\xi)=e^{\alpha|\xi|}$. We start by observing that, by our choice of $\alpha$, the function $f$ belongs to $L^p(\R,\mu_1*\mu_2)$. Indeed, by the triangular inequality, it holds
    \begin{align*}
        \norm{f}_{L^p(\R,\mu_1*\mu_2)}^p=\frac{1}{2\sqrt{2\pi}}\int_\R\int_\R e^{\alpha p|x+y|}e^{-\frac{x^2}{2}}e^{-|y|}dxydy\leq\frac{1}{2\sqrt{2\pi}}\left(\int_\R e^{\alpha p|x|-\frac{x^2}{2}}dx \right)\left(\int_\R e^{(\alpha p-1)|y|}dy \right),
    \end{align*}
    and the last term in the above chain of inequalities is finite by our choice of $\alpha$. We claim that $f$ does not belong to $\mathcal{X}_{p,q}(\R;\mu_1,\mu_2)$. Indeed 
    \begin{align*}
        \norm{f}_{p,q}^q &=\int_\R\left(\int_\R e^{\alpha p|x+y|}\frac{1}{\sqrt{2\pi}}e^{-\frac{x^2}{2}}dx\right)^{\frac{q}{p}}\frac{1}{2}e^{-|y|}dy\geq \frac{1}{2(2\pi)^{\frac{q}{2p}}}\int_\R\left(\int_0^1 e^{\alpha p|x+y|}e^{-\frac{x^2}{2}}dx\right)^{\frac{q}{p}}e^{-|y|}dy\\
        &\geq \frac{1}{2(2\pi e)^{\frac{q}{2p}}}\int_0^\infty\left(\int_0^1 e^{\alpha p|x+y|}dx\right)^{\frac{q}{p}}e^{-|y|}dy= \frac{1}{2(2\pi e)^{\frac{q}{2p}}}\int_0^\infty\left(\int_y^{1+y} e^{\alpha pz}dz\right)^{\frac{q}{p}}e^{-|y|}dy\\
        &= \frac{1}{2}\left(\frac{(e^{\alpha p}-1)^2}{2\pi e(\alpha p)^2}\right)^{\frac{q}{2p}}\int_0^\infty e^{(\alpha q-1)|y|}dy.
    \end{align*}
    The last integral is infinite by our choice of $\alpha$. So $\mathcal{X}_{p,q}(\R;\mu_1,\mu_2)\subsetneq L^p(\R,\mu_1*\mu_2)$. 
    
    To prove that $L^q(\R,\mu_1*\mu_2)\subsetneq \mathcal{X}_{p,q}(\R;\mu_1,\mu_2)$ consider the function $g(\xi)=\frac{1}{|\xi|^\beta}$ with $\frac{1}{q}\leq \beta<\frac{1}{p}$. Similar calculations as those in Example \ref{ex_Wang1} give our claim.
\end{example}

\section{Main results}\label{sect_MainResults}

In this section, we state and prove the main results of the paper. To this end, we work in a more general setting, namely that of skew operators. We begin by recalling the basic definitions and properties and we refer to \cite{AvN15} for a more comprehensive survey on this subject.

\begin{defn}\label{4.1}
Let $\mu_1$ and $\mu_2$ be two Radon probability measures on the Banach spaces $E_1$ and $E_2$, respectively, and suppose that $\hat{\mu}_2(x^*) \neq 0$ for all $x^* \in E_2^*$. A Borel measurable map $T : E_1 \to E_2$ is called a \emph{skew map} with respect to the pair $(\mu_1, \mu_2)$ if there exists a Radon probability measure $\rho$ on $E_2$ such that
\begin{equation} \label{skew}
(\mu_1 \circ T^{-1}) * \rho = \mu_2.
\end{equation}
In this case, we say that $\mu_2$ is the \emph{skew convolution product} (with respect to $T$) of $\mu_1$ and $\rho$. If $T$ is also a bounded linear operator from $E_1$ to $E_2$, then $T$ is called a \emph{skew operator} with respect to $(\mu_1, \mu_2)$. If $E_1 = E_2$ and $\mu_1 = \mu_2 = \mu$, we simply say that $T$ is a \emph{skew map} with respect to $\mu$. If $T$ is also linear and bounded, we call it a skew operator with respect to $\mu$.
\end{defn}

\begin{rmk}
We wish to list here some easy consequences and observations of Definition \ref{4.1}.
\begin{enumerate}[(i)]
    \item Condition $\hat{\mu}_2(x^*) \neq 0$ holds, for instance, when $\mu_2$ is infinitely divisible (see \cite[Proposition 5.1.1]{Lin86}).
    
    \item If the measure $\rho$, in \eqref{skew}, exists then it is unique and it is called skew convolution factor associated with $T$ and the pair $(\mu_1, \mu_2)$.
    
    \item If $E_1 = E_2 = E$ and $T = {\rm Id}_E$, then \eqref{skew} is satisfied with $\rho = \delta_0$, the Dirac measure concentrated at the origin.
\end{enumerate}
\end{rmk}

\begin{thm} \label{Lp}
Let $E$ be a separable Banach space and $\mu$ a Radon measure on $\mathcal{B}(E)$ such that $\hat{\mu}(x^*) \neq 0$ for all $x^* \in E^*$. Let $T : E \to E$ be a skew map with respect to $\mu$ and let $\rho$ be the associated skew convolution factor. Then the linear operator $P_T : B_b(E) \to B_b(E)$ defined by
\begin{equation}\label{meh_sem}
P_T f(x) := \int_E f(T(x) + y) \, d\rho(y),
\end{equation}
extends uniquely to a linear contraction $P_T : L^p(E, \mu) \to L^p(E, \mu)$ for every $p\in [1,\infty)$.
\end{thm}

\begin{proof}
Fix $p \in [1, \infty)$ and $f \in B_b(E)$. By Jensen's inequality,
\begin{align*}
\|P_T f\|_{L^p(E, \mu)}^p &= \int_E |P_T f(x)|^p d\mu(x)
\leq \int_E \int_E |f(T(x) + y)|^p d\rho(y) d\mu(x) \\
&= \int_E \int_E |f(w + y)|^p d\rho(y) d(\mu\circ T^{-1})(w) \\
&= \int_E |f(z)|^p \, d((\mu\circ T^{-1})* \rho)(z)
= \int_E |f(z)|^p \, d\mu(z)
= \|f\|_{L^p(E, \mu)}^p,
\end{align*}
which proves the claim.
\end{proof}

In order to proceed we need to recall some basic properties and definitions about Gaussian and infinitely divisible non-Gaussian Radon measures.

\subsection{Gaussian measure}

If $\gamma=\mathcal{N}(0,Q)$ then we recall that the covariance operator $Q:E^*\ra E$ of $\gamma$ satisfies 
\begin{align}\label{defn_covarianza}
    Qx^* = \int_E \langle x,x^*\rangle_{E\times E^*} x\, d\gamma(x), \qquad x^* \in E^*.
\end{align}
This integral is known to be absolutely convergent in $E$ and $Q$, defined in \eqref{defn_covarianza}, is a bounded and linear operator which is symmetric and non-negative. Namely $\langle Q x^*, y^*\rangle_{E\times E^*}=\langle Q y^*, x^*\rangle_{E\times E^*}$, for every $x^*,y^*\in E^*$, and $\langle Q x^*, x^*\rangle_{E\times E^*} \geq0$, for every $x^*\in E^*$. 
We define the Cameron--Martin space $H$ of $\gamma$ in $E$ as
the completion of the range of $Q$ with respect to the inner product defined by
\begin{align*}
\langle Qx^*,Qy^*\rangle_H:=\langle Qx^*,y^*\rangle_{E\times E^*}, \qquad\;\, x^*, y^* \in E^*.
\end{align*} 

For an in depth discussion about Gaussian measures a
their associated Cameron--Martin space, we refer to \cite{Bog98,DaPZ92,Kuo75}.



We recall that if $T\in\mathcal{L}(E)$ is a skew operator with respect to $\gamma$, then the skew convolution factor $\rho_\gamma$ is a Gaussian measure as well, and the operator $P_T$, introduced in \eqref{meh_sem}, is extensively studied in the literature, in particular in relation with the second quantization operator (see, for example, \cite{ADF25,AvN14,CMG96}).

In what follows we need the following results concerning skew maps with respect to Gaussian measures (see \cite[Propositions 3.1 and 3.2]{AvN15}) and hypercontractivity property of $P_T$ proved by Nelson in \cite[Lemma 2]{Nel73}.

\begin{pro}\label{skew_gamma}
Let $\gamma$ be a centered Gaussian measure on a separable Banach space $E$, and let $H$ be its Cameron--Martin space.
\begin{enumerate}[(i)]
\item If $T \in \mathcal{L}(E)$ is a skew operator with respect to $\gamma$, then the restriction $T_{|_{H}}:H\ra H$ of $T$ to $H$ is a contraction;
\item  If $T \in \mathcal{L}(H)$ is a contraction, then $T$ admits a linear Borel measurable extension $\overline{T} : E \to E$ such that the image measure $\gamma \circ \overline{T}^{-1}$ is Gaussian and there exists a Gaussian Radon measure $\rho$ on $E$ such that $(\gamma \circ \overline{T}^{-1}) * \rho = \gamma.$
In particular, $\overline{T}$ is a linear skew map with respect to $\gamma$.
\end{enumerate}
\end{pro}

\begin{thm}\label{thm_hyperclassical}
    Let $\gamma$ be a Gaussian measure on a separable Banach space $E$, and let $H$ be its Cameron--Martin space. If $T\in\mathcal{L}(E)$ is a skew operator with respect to $\gamma$, then the operator $P_T$, defined in \eqref{meh_sem} satisfies
    \begin{align*}
        \|P_Tf\|_{L^q(E,\gamma)}\leq\|f\|_{L^{p}(E,\gamma)},
    \end{align*}
    for every $p\in[1,\infty)$ and any $q\in [p,1+(p-1)\|T_{|_H}\|^{-2}_{\mathcal{L}(H)}]$.
\end{thm}

Note that, in the assumptions of Theorem \ref{thm_hyperclassical}, Proposition \ref{skew_gamma}(i) yields that $\|T_{|_H}\|_{\mathcal{L}(H)}\le 1$, whence $1+(p-1)\|T_{|_H}\|^{-2}_{\mathcal{L}(H)}\ge p$. Thus the interval $[p,1+(p-1)\|T_{|_H}\|^{-2}_{\mathcal{L}(H)}]$ is non-empty. 

\subsection{Hypercontractivity type estimates}

Now, we address our attention to the case when the operator $T:E\to E$ in \eqref{meh_sem}
is a skew operator with respect to an infinitely divisible Radon measure $\mu$ that, as is well known,
admits a unique representation as the convolution
\[
\mu= \delta_\xi * \gamma * \tilde{e}_s(\nu),
\]
where $\delta_\xi$ is the Dirac measure concentrated at the point $\xi\in E$, $\gamma$ is a centered Gaussian Radon measure on $E$ and $\tilde{e}_s(\nu)$ is the generalized exponential factor associated to a Radon L\'evy measure $\nu$ on $E$. (see \cite[Theorem 3.4.20]{Hey}).
Recall that a Lévy measure on $E$ is Radon measure such that
\[
\nu(\{0\})=0, \qquad \int_E \min\{1,\|x\|^2\}d\nu(x) < \infty,
\]
and the generalized exponential factor associated to $\nu$ is the measure $\tilde{e}_s(\nu)$ satifying
\begin{align*}
    \widehat{\tilde{e}_s(\nu)}(x^*)=\mathrm{exp}\left(\int_E (e^{i x^*(x)}-1)d\nu(x)\right).
\end{align*}


Our standing assumptions for this section are the following.

\begin{hyp}\label{hyp:skew-id}
Let $E$ be a separable Banach space and
\begin{enumerate}[(i)]
\item $\gamma$ is a centered Gaussian Radon measure on $E$;
\item $\nu$ is a Radon L\'evy measure on $E$ and
\begin{align*}
    \pi=\delta_\xi * \tilde{e}_s(\nu),\qquad\;\, \xi \in E
\end{align*}
(with the same notations as before);
\item $T\in\mathcal{L}(E)$ is a skew operator with respect to $\gamma*\pi$ with skew factor $\rho$.
\end{enumerate}
\end{hyp}


To begin with, we recall the following result, which is proved for instance in \cite[Proposition~2.1]{AvN15}.
\begin{pro}\label{krelly}
The following assertions are equivalent.
\begin{enumerate}[(i)]
\item $T$ is a skew map with respect to $\gamma*\pi$ with an infinitely divisible skew factor $\rho$.
\item $T$ is a skew map with respect to $\gamma$ and $\pi$ with infinitely divisible skew factors $\rho_\gamma$ and $\rho_\pi$, respectively.
\end{enumerate}
If these equivalent conditions are satisfied, then $\rho=\rho_\gamma*\rho_\pi$.
\end{pro}

Assuming either condition (i) or (ii) of Proposition \ref{krelly}, we can follow the approach of \eqref{meh_sem} to introduce the operator $P_T:B_b(E)\to B_b(E)$, given by
\begin{equation}\label{P_T}
P_Tf(x)=\int_E f(Tx+y)d\rho(y), \qquad f \in B_b(E).
\end{equation}
where $\rho=\rho_\gamma*\rho_\pi$ is the skew factor of $\gamma*\pi$, i.e. $((\gamma*\pi)\circ T^{-1})*\rho=\gamma*\pi$. In a similar way, one can define $P_{T}^\gamma$ and $P_{T}^\pi$ just replacing $\rho$ in \eqref{P_T} by $\rho_\gamma$ and $\rho_\pi$, respectively. 

Recall that, thanks to Theorem \ref{Lp}, the operators $P_T^\gamma$ and $P_T^\pi$ extend consistently to contractive operators in $L^p(E,\gamma)$ and $L^p(E,\pi)$ respectively, for any $p \in [1,\infty)$. 
Moreover, Theorem \ref{thm_hyperclassical} allows us to consider the tensor product operator $P_T^\gamma \otimes P_T^\pi$ in $L^p(E,\gamma)\otimes L^p(E, \pi)$ and prove its continuity with respect to the crossnorm $\alpha_q^{L^p(E,\gamma),L^q(E,\pi)}$ introduced in Section \ref{sect_Tensor}.

\begin{pro}\label{prop2}
For every $p \in (1,\infty)$ and every $q \in [p,1+(p-1)\|T_{|_H}\|^{-2}_{\mathcal{L}(H)}]$, the operator
$$
P_T^\gamma\otimes P_T^\pi:L^p(E,\gamma)\otimes L^q(E,\pi)\longrightarrow L^q(E,\gamma)\otimes_q L^q(E,\pi)
$$
defined on elements of the form $G=\sum_{i=1}^n f_i\otimes g_i$ by
$$
(P_T^\gamma\otimes P_T^\pi)(G)=\sum_{i=1}^n (P_T^\gamma f_i)\otimes (P_T^\pi g_i),
$$
extends uniquely to a continuous linear operator with domain $L^p(E,\gamma)\otimes_q L^q(E,\pi)$ and its norm is less than or equal to $1$. In particular if $p=q$ then $P_T^\gamma\otimes P_T^\pi$ is continuous from $L^p(E\times E,\gamma\times\pi)$ into itself.
\end{pro}

\begin{proof}
The result follows from the bounded extension theorem for dense subspaces. Indeed, given $G$ as in the statement and $\varepsilon > 0$, there exist $n_\varepsilon \in \mathbb{N}$, $f_{i,\varepsilon} \in L^p(E,\gamma)$, and $g_{i,\varepsilon} \in L^q(E,\pi)$, such that
$$
G = \sum_{i=1}^{n_\varepsilon} f_{i,\varepsilon} \otimes g_{i,\varepsilon},
$$
and
$$
\left(\sum_{i=1}^{n_\varepsilon}\|f_{i,\varepsilon}\|^q_{L^p(E,\gamma)}\right)^{\frac{1}{q}} \mu_{q',L^q(E,\pi)}(g_{1,\varepsilon},\ldots,g_{n_\varepsilon,\varepsilon}) \leq \alpha_q^{L^p(E,\gamma),L^q(E,\pi)}(G) + \varepsilon.
$$
Using the contractivity of $P_T^\pi$ on $L^q(E,\pi)$, the hypercontractivity of $P_T^\gamma$ (Theorem \ref{thm_hyperclassical}) and the definition of the Chevet--Saphar norm, we obtain
\begin{align*}
\alpha_q^{L^q(E,\gamma),L^q(E,\pi)}((P_T^\gamma \otimes P_T^\pi)(G))
& \leq \left(\sum_{i=1}^{n_\varepsilon} \|P_T^\gamma f_{i,\varepsilon}\|^q_{L^q(E,\gamma)}\right)^{\frac{1}{q}} \mu_{q',L^q(E,\pi)}(P_T^\pi g_{1,\varepsilon}, \ldots, P_T^\pi g_{n_\varepsilon,\varepsilon}) \\
& \leq \left(\sum_{i=1}^{n_\varepsilon} \|f_{i,\varepsilon}\|^q_{L^p(E,\gamma)}\right)^{\frac{1}{q}} \mu_{q',L^q(E,\pi)}(g_{1,\varepsilon}, \ldots, g_{n_\varepsilon,\varepsilon}) \\
& \leq \alpha_q^{L^p(E,\gamma),L^q(E,\pi)}(G) + \varepsilon.
\end{align*}
Since $\varepsilon > 0$ is arbitrary, the claim follows.
\end{proof}

Arguing similarly one can obtain a version of the previous result where the roles of $\pi$ and $\gamma$ are switched.
\begin{pro}\label{pro_girata}
For every $r \in (1,\infty)$ and $p \in [r,1+(r-1)\|T_{|_H}\|^{-2}_{\mathcal{L}(H)}]$, the operator
$$
P_T^\pi\otimes P_T^\gamma:L^r(E,\pi)\otimes L^r(E,\gamma)\longrightarrow L^r(E,\pi)\otimes_r L^p(E,\gamma)
$$
defined on elements of the form $G=\sum_{i=1}^n f_i\otimes g_i$ by
$$
(P_T^\pi\otimes P_T^\gamma)(G)=\sum_{i=1}^n (P_T^\pi f_i)\otimes (P_T^\gamma g_i),
$$
extends uniquely to a continuous linear operator with domain $L^r(E,\pi)\otimes_r L^r(E,\gamma)$ and its norm is less than or equal to $1$.
\end{pro}


In the following lemma, through the map $F_\bullet$, introduced in Proposition \ref{pro_Fbullet}, we establish a connection between $F_{P_Tf}$ and $(P_T^\gamma \otimes P_T^\pi)F_f$ whenever $f$ belongs to $\mathcal{X}_{p,q}(E;\gamma,\pi)$. A similar result has been proved in \cite{AvN15} for functions $f \in L^2(\gamma*\pi)$.

\begin{lemma}\label{ext_p}
Let $p,q \in [1,\infty)$ with $q\ge p$. Then, for every $f \in \mathcal{X}_{p,q}(E;\gamma,\pi)$ it holds
\begin{equation}\label{FPTf}
F_{P_T f}=(P_T^\gamma \otimes P_T^\pi)F_f,\qquad (\gamma\times\pi)\text{-a.e. in\,} E\times E.
\end{equation}
\end{lemma}

\begin{proof}
We point out that equality \eqref{FPTf} is meaningful since we can identify $L^q(E,\pi;L^p(E,\gamma))$ with a subspace of $L^p(E,\gamma)\otimes_q L^q(E,\pi)$, as pointed out in Remark \ref{dueo}.

Now, let $f \in \mathcal{X}_{p,q}(E;\gamma,\pi)$. Since $F_f \in X_{p,q}(\gamma,\pi)\subset L^q(E,\pi;L^p(E,\gamma))\subset L^p(E,\pi;L^p(E,\gamma))$, by Propositions \ref{prop1} and \ref{pro_Fbullet}, there exists $(g_n) \subseteq B_b(E)\otimes B_b(E)$ such that $g_n \to F_f$ in $L^p(E\times E,\gamma\times\pi)$.  
Moreover, for $(\gamma\times \pi)$-almost every $(x, y) \in E\times E$, if $g_n= g_n^{(1)}\otimes g_n^{(2)}$, then
\begin{align*}
[(P_T^\gamma \otimes P_T^\pi)g_n](x,y)
& = (P_T^\gamma g_n^{(1)})(x)  (P_T^\pi g_n^{(2)})(y) \\
& = \left( \int_E g_n^{(1)}(Tx+z_1)\,d\rho_\gamma(z_1) \right) \left( \int_E g_n^{(2)}(Ty+z_2)\,d\rho_\pi(z_2) \right) \\
& = \int_E \int_E (g_n^{(1)} \otimes g_n^{(2)})(Tx+z_1, Ty+z_2)\,d\rho_\gamma(z_1)\,d\rho_\pi(z_2)\\
&= \int_E \int_E g_n(Tx+z_1, Ty+z_2)\,d\rho_\gamma(z_1)\,d\rho_\pi(z_2)=:G_n(x,y).
\end{align*}
By using the continuity of $P_T^\gamma \otimes P_T^\pi$ in $L^p(E\times E, \gamma \times \pi)$ (see Proposition \ref{prop2}), we deduce that
$(P_T^\gamma \otimes P_T^\pi)g_n$ converges to $(P_T^\gamma \otimes P_T^\pi)F_f$ in $L^p(E\times E,\gamma\times \pi)$ as $n \to +\infty$. On the other, we claim that $G_n$ converges to 
$$G(x,y) := \int_E \int_E F_f(Tx+z_1, Ty+z_2)\,d\rho_\gamma(z_1)\,d\rho_\pi(z_2)$$
in $L^p(E\times E,\gamma \times \pi)$. Indeed, by the Jensen inequality and the skew property of $T$, we get
\begin{align*}
\lim_{n\ra\infty}&\|G_n - G\|_{L^p(E\times E,\gamma \times \pi)} 
= \lim_{n\ra\infty}\int_E\int_E |G_n(x,y) - G(x,y)|^p\,d\gamma(x)\,d\pi(y) \\
&\leq \lim_{n\ra\infty}\int_E\int_E\int_E\int_E |g_n(Tx+z_1, Ty+z_2) - F_f(Tx+z_1, Ty+z_2)|^p\,d\rho_\gamma(z_1)\,d\rho_\pi(z_2)\,d\gamma(x)\,d\pi(y) \\
&= \lim_{n\ra\infty}\int_{E}\int_E |g_n(x,y) - F_f(x,y)|^p\,d\gamma(x)\,d\pi(y) = 0.
\end{align*}
Consequently, by uniqueness, for $(\gamma\times\pi)$-almost every $(x,y)\in E\times E$ we get 
\begin{align*}
[(P_T^\gamma \otimes P_T^\pi)F_f](x,y) 
&= \int_E \int_E F_f(Tx+z_1, Ty+z_2)\,d\rho_\gamma(z_1)\,d\rho_\pi(z_2) \\
&= \int_E f(Tx + Ty + z)\,d(\rho_\gamma * \rho_\pi)(z) \\
&= P_T f(x+y) = F_{P_T f}(x,y),
\end{align*}
which concludes the proof.
\end{proof}

\begin{rmk}\rm{ Note that in the previous lemma we can obtain a similar result if we change the roles of $\gamma$ and $\pi$ and simultaneously those of $p$ and $q$. Consequently, for every $q,p \in [1,\infty)$ with $p\ge q$ and any $f \in \mathcal{X}_{q,p}(E;\pi, \gamma)$ it holds
\begin{equation}\label{FPTf_girata}
F_{P_T f}=(P_T^\pi \otimes P_T^\gamma)F_f,\qquad (\pi\times\gamma)\text{-a.e. in\,} E\times E.
\end{equation}}
\end{rmk}
\noindent
Proposition \ref{prop2} and Lemma \ref{ext_p} are crucial to prove the main result of this section which establishes two different versions of a hypercontractivity type property. 
\begin{thm}\label{thm_hyper}
Let $p \in (1,\infty)$. For every $q \in [p,1+(p-1)\|T_{|_H}\|^{-2}_{\mathcal{L}(H)}]$  and  $f \in \mathcal{X}_{p,q}(E;\gamma, \pi)$ it holds that
\begin{equation}\label{est-qpq}
\|P_T f\|_{L^q(E, \gamma*\pi)}\le \|f\|_{p,q}.
\end{equation}
\noindent
Moreover, if $r\in (1,\infty)$ and $f \in L^r(E, \gamma*\pi)$ then $P_T f\in \mathcal{X}_{p,r}(E;\gamma,\pi)$ for every $p\in [r,1+(r-1)\|T_{|_H}\|^{-2}_{\mathcal{L}(H)}]$ and 
\begin{equation}\label{idea_davide}
\|P_T f\|_{p,r}\le \|f\|_{L^r(E,\gamma*\pi)}.
\end{equation}

\end{thm}

\begin{proof}
Let us start by proving \eqref{est-qpq}. So, let us fix $p,q$ as in the statement.
Applying \eqref{formula_in_proof} with $g=P_Tf$ and $r=q$ we get
$$\|P_T f\|_{L^q(E, \gamma*\pi)}=\|F_{P_T f}\|_{L^q(E\times E, \gamma\times \pi)}.$$
Thus, using Lemma \ref{ext_p} and Propositions \ref{prop1} and \ref{prop2}, we can estimate
\begin{align*}
\|P_T f\|_{L^q(E, \gamma*\pi)}=&\|F_{P_T f}\|_{L^q(E\times E, \gamma\times \pi)}\\
=& \|(P_T^\gamma \otimes P_T^\pi)F_f\|_{L^q(E\times E, \gamma\times \pi)}\\
=& \|(P_T^\gamma \otimes P_T^\pi)F_f\|_{L^q(E,\gamma)\otimes_q L^q( E,\pi)}\\
\le & \|F_f\|_{L^p(E,\gamma)\otimes_q L^q( E,\pi)}\\
\le &  \|F_f\|_{L^q(E,\pi;L^p(E, \gamma))}\\
= &\|f\|_{p,q}
\end{align*}
where in the second to last inequality we used the embedding of $L^q(E,\pi;L^p(E, \gamma))$ into $L^p(E,\gamma)\otimes_{q} L^q(E,\pi)$ stated in Proposition \ref{prop1}(i). In order to prove \eqref{idea_davide}, let us fix $f \in L^r(E, \gamma*\pi)$ for some $r\in (1,\infty)$. Arguing similarly as above and using equality \eqref{FPTf_girata} and Proposition \ref{pro_girata} we get
\begin{align*}
    \|P_Tf\|_{p,r} & = \|F_{P_T f}\|_{L^r(E,\pi;L^p(E,\gamma))}\\
    &\leq \|F_{P_T f}\|_{L^r(E,\pi)\otimes_r L^p(E,\gamma)}\\
    & =\|(P_T^\pi\otimes P_T^\gamma)F_f\|_{L^r(E,\pi)\otimes_r L^p(E,\gamma)}\\
    & \leq \|F_f\|_{L^r(E,\pi)\otimes_r L^r(E,\gamma)}\\
    & = \|F_f\|_{L^r(E,\pi; L^r(E,\gamma)}\\
    & = \|f\|_{L^r(E,\gamma*\pi)}.
\end{align*}
Thus, estimate \eqref{idea_davide} is proved.
\end{proof}

\begin{cor}
For every $p \in (1,\infty)$ and $q \in [p, 1+(p-1)\|T_{|_H}\|_{\mathcal{L}(H)}^{-2}]$, the operator $P_T$ defined in \eqref{P_T} extends uniquely to a contractive linear operator from $\mathcal{X}_{p,q}(E;\gamma, \pi)$ into itself.
\end{cor}
\begin{proof}
The continuous embeddings in Proposition \ref{pro_spacesXpq} and estimate \eqref{est-qpq} yield that 
\begin{equation*}
\|P_T f\|_{p,q}\le \|P_Tf\|_{L^q(E,\gamma*\pi)}\le \|f\|_{p,q},\qquad\;\, f \in B_b(E).
\end{equation*}
The density of $B_b(E)$ in $\mathcal{X}_{p,q}(E;\gamma,\pi)$, stated in Proposition \ref{pro_spacesXpq}, concludes the proof.
\end{proof}

\subsection{An application to compactness}\label{consequences}

A classical consequence of the hypercontractivity property of an operator is the compactness of its square (see, for example, \cite{Wu2000}). Theorem \ref{thm_hyper} provides a similar type of consequence. We start by observing that bounded sets in $\mathcal{X}_{p,q}(E;\gamma,\pi)$ are $p$-uniformly integrable. 

\begin{lemma}\label{cor_p-unif_int}
    Let $p,q \in (1,\infty)$ be such that $p\in[1+(q-1)\|T_{|_H}\|^{2}_{\mathcal{L}(H)},q)$. For every $M\in\R^+$ the set
    \begin{align*}
       \Gamma_{p,q}(M):= \{P_Tf\,|\,\|f\|_{p,q}\leq M\}\subseteq L^{p}(E,\gamma*\pi)
    \end{align*}
    is $p$-uniformly integrable, namely, for every $\varepsilon>0$, there exists $\delta>0$ such that for every measurable set $A\subseteq E$ with $(\gamma*\pi)(A)\le \delta$ it holds that
    $$\sup_{g \in \Gamma_{p,q}(M)}\int_A |g|^pd(\gamma*\pi)\le \varepsilon.$$
\end{lemma}

\begin{proof}
   Let us fix $p, q$ as in the statement and for every $\eps>0$ consider a measurable set $A\subseteq E$ such that $(\gamma*\pi)(A)\leq \left(\frac{\eps}{M^p}\right)^{\frac{q}{q-p}}$. Thus, by the H\"older inequality we get
    \begin{align*}
        \int_A|P_Tf|^pd(\gamma*\pi) & = \int_E\chi_A|P_Tf|^pd(\gamma*\pi)\leq  [(\gamma*\pi)(A)]^{\frac{q-p}{q}}\|P_Tf\|_{L^q(E,\gamma*\pi)}^p\\
        &\leq [(\gamma*\pi)(A)]^{\frac{q-p}{q}}\|f\|_{p,q}^p\leq [(\gamma*\pi)(A)]^{\frac{q-p}{q}} M^p \leq \eps.
    \end{align*}
    This concludes the proof.
\end{proof}

Taking Lemma \ref{cor_p-unif_int} into account and arguing as in \cite[Theorem 2.3]{Wu2000} we deduce the compactness of $P^2_T$ as operator from $\mathcal{X}_{p,q}(E;\gamma,\pi)$ into $L^p(E, \gamma*\pi)$.

\begin{thm}
    Assume that there exist a measure $\vartheta\gg \gamma*\pi$ and a function $k:E\times E\ra\R$ such that for every $x\in E$ the map $y\mapsto k(x,y)$ belongs to $L^1(E,\vartheta)$ and for $(\gamma*\pi)$-a.e. $x\in E$ and $f\in B_b(E)$ it holds 
    \begin{align*}
        P_Tf(x)=\int_Ek(x,y)f(y)d\vartheta(y).
    \end{align*}
    Then, for every $p,q \in (1,\infty)$ such that $p\in[1+(q-1)\|T_{|_H}\|^{2}_{\mathcal{L}(H)},q)$, the operator $P_T^2=P_T\circ P_T:\mathcal{X}_{p,q}(E;\gamma,\pi)\ra L^p(E,\gamma*\pi)$ is compact.
\end{thm}

\section{Generalized Mehler semigroups}\label{section_Mehler}

In this section we focus our attention to a particular case of skew operators, that of generalized Mehler semigroups, extensively studied in the literature. We refer to \cite{ARW00,AFP23,BRS96,FR00,LR02,LR04,Lun22,LR21,OR16,Pes11,RW03,SS01} for an overview of their properties.

We consider generalized Mehler semigroups defined on bounded Borel functions $f: E \to \mathbb{R}$, where $E$ is a separable Hilbert space endowed with its inner product $\langle \cdot, \cdot \rangle_E$ and the associated norm $\norm{\cdot}_E$. When no confusion arises, we omit the subscript.
A family of operators 
$(P(t))_{t \ge 0} \subseteq \mathcal{L}(B_b(E))$
is called a \emph{generalized Mehler semigroup} if there exist a strongly continuous semigroup 
$(S(t))_{t \ge 0}$ on $E$ and a family of probability measures 
$(\mu_t)_{t \ge 0}$ on $E$ such that
\begin{equation}\label{generalized_mehler}
(P(t) f)(x) = \int_E f(S(t) x + y)d\mu_t(y),
\qquad f \in B_b(E),\ x \in E,\ t \ge 0,
\end{equation}
and satisfying the skew convolution equation
\begin{equation}\label{skew_conv}
\mu_{t+r} = (\mu_t \circ [S(r)]^{-1})*\mu_r, \qquad r,t \ge 0.
\end{equation}
Condition \eqref{skew_conv} is necessary and sufficient for $(P(t))_{t\geq 0}$ to be a semigroup and it says that $S(r)$ is a skew operator with respect to $(\mu_t,\mu_{t+r})$ for every $t,r> 0$. We recall that if for any $\xi\in E$
the function $t\mapsto\widehat{\mu_t}(\xi)$ is absolutely continuous on $[0,\infty)$ and differentiable
at $t=0$ then, setting
\[
\lambda(\xi):=-\frac{d}{dt}\widehat{\mu_t}(\xi)_{|_{t=0}},
\]
the function
$t\mapsto \lambda((S(t))^*\xi)$ belongs to $L^1_{\rm loc}((0,\infty))$ (see \cite[Lemma 2.6]{BRS96}), hence \eqref{skew_conv}
is equivalent to
\begin{align*}
\hat{\mu}_t(\xi)= \exp\left(-\int_0^t \lambda ((S(r))^* \xi)dr\right), \qquad t\ge 0,\ \xi\in E .
\end{align*}

In this case, the function $\lambda : E \to \C$ is a continuous negative definite function with $\lambda(0)=0$. It is called \emph{the characteristic exponent} of $P(t)$, and, assuming that it is Sazonov continuous (see \cite[Chapter 8.6]{Lin86}), it can be proved that it has an unique Lévy--Khintchine representation (see \cite[Theorem VI.4.10]{Par67}) as
\begin{equation}\label{lambda}\lambda(\xi)= -i\langle b, \xi \rangle + \frac{1}{2}\langle Q\xi, \xi \rangle + \int_E \Big( 1 - e^{i\langle \xi, y \rangle} + i\langle \xi, y \rangle \chi_{\{\|y\|\le1\}} \Big)d\nu(y),
\end{equation}
where $b\in E$, $Q$ is a nonnegative symmetric trace-class operator on $E$ and $\nu$ is a Lévy measure on $\mathcal{B}(E)$.
In the sequel we write $\lambda=[b,Q,\nu]$ to associate the triple $[b,Q,\nu]$ with $\lambda$ according to \eqref{lambda}. Even if $(P(t))_{t\geq 0}$ preserves $C_b(E)$ and $\|P(t)\|_{\mathcal{L}(C_b(E))}\le 1$ for every $t>0$, in general $P(t)$ is not strongly continuous in $C_b(E)$. However, the continuity of the map $(t,x)\mapsto (P(t)f)(x)$, for $f \in C_b(E)$ allows to define the weak generator $\mathcal{L}$ of $P(t)$ through its resolvent. The action of $\mathcal{L}$ on smooth functions $f \in \mathcal{F}C^2_b(E)$ is described by the formula
\begin{equation}\label{formula_L}
[\mathcal{L}f](x)=\frac{1}{2}{\rm Tr}(QD^2f(x))+
\langle Ax+b, D f(x)\rangle+\int_X[f(x+y)-f(x)-\gen{Df(x),y}\chi_{B_1}(y)]d\nu(y)
\end{equation}
where $A:D(A)\subseteq E\to E$ denotes the infinitesimal generator of $S(t)$.
Note that for regular functions, the integral in \eqref{formula_L} is
well defined by the Taylor expansion and the fact that $\nu$ is a L\'evy measure.
To go further, we recall some sufficient conditions provided in \cite[Theorem 3.1]{FR00} that guarantee the existence of a unique invariant measure for $P(t)$, namely a Borel probability measure $\mu$ on $E$ such that
\begin{align*}
\int_E P(t)f d\mu=\int_E f d\mu,\qquad\;\, t \ge 0,\, f \in B_b(E),
\end{align*}
or equivalently $\mu=(\mu\circ (S(t))^{-1})*\mu_t$ for every $t >0$ stating that $S(t)$ is a skew operator with respect to $\mu$, see \eqref{skew}.
If $\lambda=[b,Q,\nu]$ then $\mu_t$ is infinitely divisible for every $t>0$ with characteristic exponent described by the triple $[b_t,Q_t,\nu_t]$ equivalently $\mu_t=\delta_{b_t}*\gamma
(0,Q_t)*\nu_t$
where
\begin{equation}\label{b_t}
b_t:=\int_0^t S(r) b\, dr+\int_0^t\int_E S(r) x\Big(\chi_{B_1}(S(r) x)-\chi_{B_1}(x)\Big)\nu(dx)dr,
\end{equation}
\begin{equation}\label{Qt}
Q_t:= \int_0^t S(r)Q(S(r))^*dr
\end{equation}
and the L\'evy measures $\nu_t$ defined setting $\nu_t(\{0\})=0$ and
\begin{equation}\label{defM_t}
\nu_t(B):=\int_0^t \nu([S(r)]^{-1}(B))dr, \qquad B \in \mathcal{B}(E), 0\notin B.
\end{equation}
Sufficient conditions that ensure the existence of a unique invariant measure for $P(t)$ are stated here below.
\begin{hyp}\label{Hyp_FR}
Let $b\in E$, $Q$ a nonnegative symmetric trace-class operator on $E$ and $\nu$ a Lévy measure on $\mathcal{B}(E)$ such that $\lambda=[b,Q,\nu]$ and let $(S(t))_{t \ge 0}$ be a strongly continuous semigroup on $E$. Let $b_t, Q_t$ and $\nu_t$ defined as in \eqref{b_t}, \eqref{Qt} and \eqref{defM_t}. Assume that
\begin{enumerate}[(i)]
\item 
there exists $b_\infty:=\lim_{t \to \infty}b_t$ in $E$;
\item $\sup_{t>0}\operatorname{Tr} Q_t=:\operatorname{Tr} Q_\infty<\infty$;
\item 
setting $\nu_\infty:=\sup_{t>0} \nu_t$ (i.e., $\nu_\infty(\{0\})=0$ and
$\nu_\infty(B)=\int_0^\infty \nu([S(r)]^{-1}(B))dr$, $B\in \mathcal{B}(E)$, $0\notin B$), it holds that
\begin{equation*}
\int_E  \min\{1,\|y\|_E^2\}d\nu_\infty(y)<\infty;
\end{equation*}
\item $ \lim_{t \to \infty}S(t) x=0$ in $E$ for every $x \in E$.
\end{enumerate}
\end{hyp}

\begin{rmk}{\rm In \cite[Theorem 6.7]{Chojnowska87}, in the case when $S(t)$ is exponentially stable, the authors show that the condition 
\begin{equation}\label{log_Levy}
\int_{\{\|x\|_E\geq 1\}}\ln\|x\|_E\, d\nu(x)<\infty
\end{equation}
where $\nu$ is the L\'evy measure in \eqref{lambda}}, is sufficient to guarantee the existence of a unique invariant measure for $P(t)$.
\end{rmk}

\begin{thm}\label{sunto}
Under Hypotheses \ref{Hyp_FR}, the measure $\mu=\delta_{b_\infty}*\gamma(0,Q_\infty)*\nu_\infty$ is the unique invariant measure for $P(t)$. The semigroup $P(t)$ can be extended to a contractive strongly continuous semigroup $($still denoted by $P(t)$$)$ on $L^p(E, \mu)$ for any $1\le p<\infty$. 
\end{thm}

Note that in this case Hypothesis \ref{hyp:skew-id} is satisfied with $\gamma:=\gamma(0, Q_\infty)$, $\pi:=\delta_{b_\infty}*\nu_\infty$ and $T=S(t)$, for every $t>0$.
Consequently Theorems \ref{Lp} and \ref{thm_hyper}  apply and the consequences stated in Section \ref{consequences} can be deduced. In particular Theorem \ref{thm_hyper} can be reformulated as follows.

\begin{thm}\label{hyper_appl}
Assume that Hypotheses \ref{Hyp_FR} are satisfied. Then, for every $t>0$, $p \in (1,\infty)$, $q \in [p,1+(p-1)\|S(t)_{|_H}\|^{-2}_{\mathcal{L}(H)}]$  and  $f \in \mathcal{X}_{p,q}(E;\gamma, \pi)$, $P(t)f$ belongs to $L^q(E, \gamma*\pi)$ and
\begin{equation}\label{est-qpq_semigruppo}
\|P(t) f\|_{L^q(E, \gamma*\pi)}\le \|f\|_{p,q},\qquad\;\, t>0.
\end{equation}
Moreover, if $r\in (1,\infty)$ and $f \in L^r(E, \gamma*\pi)$ then $P(t) f\in X_{p,r}(\gamma,\pi)$ for every $t>0$ and $p\in [r,1+(r-1)\|S(t)_{|_H}\|^{-2}_{\mathcal{L}(H)}]$. Moreover 
\begin{equation}\label{idea_davide_semigruppo}
\|P(t) f\|_{p,r}\le \|f\|_{L^r(E,\gamma*\pi)}.
\end{equation}
\end{thm}
\begin{rmk}{\rm Note that in the Gaussian case ($\pi=\delta_0$), inequalities \eqref{est-qpq_semigruppo} and \eqref{idea_davide_semigruppo} are equivalent and state that the Ornstein-Uhlenbeck semigroup $R(t)$, introduced in \eqref{OU_Classical}, is hypercontractive, i.e. for every $p>1$ and $t>0$, $R(t)$ maps $L^p(E, \gamma)$ into $L^q(E, \gamma)$ for $q \in [p,1+(p-1)\|S(t)_{|_H}\|^{-2}_{\mathcal{L}(H)}]$ and
\begin{equation}\label{hyp_classica}
\|R(t)f\|_{L^q(E, \gamma)}\le \|f\|_{L^p(E, \gamma)}, \qquad\;\, t>0, f \in L^p(E,\gamma).
\end{equation}}
\end{rmk}

Now, we provide two examples of generalized Mehler semigroups introduced in \cite{RW03} that allow us to comment our hypercontractive type estimates. As they provide key insights pertinent to our results, we revisit the first of them in details. For the second one we refer to \cite[Example 3.3]{RW03}.

\begin{example}\label{example_1_Wang}
Let $E = \R^n$, $S(t) = e^{-\beta t}\mathrm{Id}_{\R^n}$ for every $t>0$ and some $\beta>0$. For every $\xi \in \R^n$ we define 
\(
\lambda(\xi) := \delta \|\xi\|^2 + \|\xi\|^\alpha,
\)
for some $\delta \ge 0$, and $\alpha \in (0,2)$. 
The associated generalized Mehler semigroup is given by
\[
[P(t) f](x) := \int_{\R^n} f(e^{-\beta t}x + y)\, d\mu_t(y), 
\qquad f \in B_b(\R^n),
\]
where
\begin{align*}
        \hat{\mu_t}(\xi)={\rm exp}\left[-\frac{1}{\alpha\beta}(1-e^{-\alpha\beta t})\norm{\xi}^\alpha-\frac{\delta}{2\beta}(1-e^{-2\beta t})\norm{\xi}^2\right].
    \end{align*}
In this framework, \cite[Theorem 3.1]{FR00} yields that there exists a unique invariant measure $\mu$ for $P(t)$ such that its characteristic function satisfies the equality
    \begin{align*}
        \hat{\mu}(\xi)={\rm exp}\left[-\frac{1}{\alpha\beta}\norm{\xi}^\alpha-\frac{\delta}{2\beta}\norm{\xi}^2\right].
    \end{align*} 
    
Consequently, Theorem \ref{hyper_appl} can be applied and estimates \eqref{est-qpq_semigruppo} and \eqref{idea_davide_semigruppo} are satisfied with $\gamma$ equal to the Gaussian measure with characteristic function 
\(
\widehat{\gamma}(\xi) = \exp\left(-(2\beta)^{-1}\delta\|\xi\|^2\right)
\)
and $\pi$ the infinitely divisible measure with characteristic function 
\(
\widehat{\pi}(\xi) = \exp\left(-(\alpha\beta)^{-1}\|\xi\|^\alpha\right).
\)
However, for every $t>0$ and $1<p<q<\infty$, it holds
\begin{equation}\label{nohyp}
\|P(t)\|_{\mathcal{L}(L^q(\gamma * \pi), L^p(\gamma * \pi))} = \infty.
\end{equation}
Indeed, following the arguments of \cite[Example 3.2]{RW03}, it holds that both $\mu_t$ and $\mu$ are absolutely continuous with respect to the $n$-dimensional Lebesgue measure and there exists a positive constant $C$ such that 
    \begin{align}\label{oax}
        \frac{1}{C(1+\norm{x}^{n+\alpha})}\leq\eta_t(x):=\frac{d\mu_t}{dx}(x)\leq\frac{C}{1+\norm{x}^{n+\alpha}},\qquad x\in\R^n.
    \end{align}
    The same estimate holds for $\eta:=\frac{d\mu}{dx}$. Let $\alpha<p$ and for $\eps\in(\frac{\alpha}{q},\frac{\alpha}{p})$, consider the function $f(x)=\norm{x}^\eps$. Simple calculations give that $f\in L^p(\R^n,\mu)$, and by \eqref{oax} it holds for every $x\in\R^n$ with $\|x\|\geq 2e^{t\beta}$
    \begin{align*}
        [P(t)f](x) &=\int_{\R^n}\|e^{-\beta t}x+y\|^\eps\eta_t(y)dy\geq \frac{1}{C}\int_{\R^n}\frac{\|e^{-\beta t}x+y\|^\eps}{1+\norm{y}^{n+\alpha}}dy\\
        &\geq \frac{1}{2C} \int_{\{\norm{y}\leq 1\}}\big|e^{-\beta t}\norm{x}-\norm{y}\big|^\eps dy.
    \end{align*}
   Using now the $n$-dimensional spherical coordinates we get that there exists a positive constant $K_n$, depending only on the dimension $n$, such that 
   \begin{align*}
       [P(t)f](x) &\geq \frac{1}{2C} \int_{\{\norm{y}\leq 1\}}\big|e^{-\beta t}\norm{x}-\norm{y}\big|^\eps dy= \frac{K_n}{2C} \int_0^1\big|e^{-\beta t}\norm{x}-\rho\big|^\eps \rho^{n-1} d\rho\\
       &= \frac{K_n}{2C} \int_0^1\big(e^{-\beta t}\norm{x}-\rho\big)^\eps \rho^{n-1} d\rho\geq \frac{K_n}{2C}(e^{-\beta t}\norm{x}-1)^\eps\int_0^1\rho^{n-1}d\rho\\
        &\geq \frac{K_n}{2nC}\left(\frac{1}{2}e^{-\beta t}\norm{x}\right)^\eps.
   \end{align*}
   So we get
   \begin{align*}
       \int_{\R^n}|[P(t)f](x)|^qd\mu(x)\geq \frac{1}{C}\left(\frac{K_n e^{-\eps\beta  t}}{2^{\eps+1}nC}\right)^q\int_{\{\max\{1,\norm{x}\}\geq 2e^{\beta t}\}}\frac{\norm{x}^{q\eps}}{1+\norm{x}^{n+\alpha}}dx.
   \end{align*}
   Again using the $n$-dimensional spherical change of variable we get
   \begin{align*}
       \int_{\R^n}|[P(t)f](x)|^qd\mu(x) &\geq \frac{K_n}{C}\left(\frac{K_n e^{-\eps\beta  t}}{2^{\eps+1}nC}\right)^q\int_{\max\{1, 2e^{\beta t}\}}^{\infty}\frac{\rho^{q\eps}}{1+\rho^{n+\alpha}}\rho^{n-1}d\rho\\
       &\geq \frac{K_n}{2C}\left(\frac{K_n e^{-\eps\beta  t}}{2^{\eps+1}nC}\right)^q\int_{\max\{1,2e^{\beta t}\}}^{\infty}\frac{1}{\rho^{1+\alpha-q\eps}}d\rho,
   \end{align*}
   the last integral diverges due to our choice of $\alpha$ and $\eps$. Estimate \eqref{nohyp} shows how, in this case, the results of Theorem \ref{thm_hyper} cannot be strengthened to yield the classical hypercontractivity for $P(t)$. 
The presence of the non-Gaussian component in $\lambda$ prevents the semigroup from exhibiting the integrability-improving behaviour typical of the purely Gaussian case.
\end{example}

To conclude this section we compare the results in Theorem \ref{hyper_appl} with the unique (to the best of our knowledge) hypercontractivity result for generalized Mehler semigroups (see \cite[Theorem 1.5]{RW03} and \cite[Theorem 5.9]{OR16} for a generalization). In  \cite[Theorem 1.5]{RW03} the authors assume that, for every $x\in E$, the measure $\mu_t\circ\theta^{-1}_{S(t)x}$ is absolutely continuous with respect to $\mu_t$ with density in $L^{p'}(E,\mu_t)$ for every $t>0$. Here $\theta_z(y)=y+z$, $y,z\in E$ and $p'$ denotes the conjugate exponent of $p$. Setting
\begin{align*}
    \Phi_{t,p'}(x)=\left\|\frac{d(\mu_t\circ\theta^{-1}_{S(t)x})}{d\mu_t}\right\|_{L^{p'}(E,\mu_t)},
\end{align*}
and assuming that there exists $\varepsilon>0$ such that
\begin{align}\label{condizione_Wang}
    C(t,p',\varepsilon)
    :=\int_E\left(\frac{1}{\int_E[\Phi_{t,p'}(x-y)]^{-p}d(\gamma*\pi)(y)}\right)^{1+\varepsilon}
    d(\gamma*\pi)(x)<\infty,
\end{align}
they prove that
\begin{align}\label{hyp-wang}
    \|P(t)f\|_{L^{p(1+\varepsilon)}(E,\gamma*\pi)}
    \le [C(t,p',\varepsilon)]^{\frac{1}{p(1+\varepsilon)}}\,
    \|f\|_{L^p(E,\gamma*\pi)}.
\end{align}
Clearly, this result is meaningless if \eqref{condizione_Wang} is not satisfied as, for instance, in Example \ref{example_1_Wang} and this prevents from deducing any summability improvement with respect to $\gamma*\pi$ by the action of $P(t)$ itself.  On the other hand, if \eqref{condizione_Wang} is satisfied and then estimate \eqref{hyp-wang} holds true, we point out that for large $t$, estimate \eqref{idea_davide_semigruppo} is  stronger than estimate \eqref{hyp-wang}. Indeed by \eqref{inclusion} we deduce that
\begin{align*}
    \mathcal{X}_{p(t),p}(E;\gamma,\pi)
    \subseteq L^{p(t)}(E,\gamma*\pi)
    \subseteq L^{p(1+\varepsilon)}(E,\gamma*\pi),
\end{align*}provided that 
\begin{align}\label{cond_t}
    t\ge \frac{1}{\omega}\ln\!\left(\frac{M^2(p(1+\varepsilon)-1)}{p-1}\right)
\end{align}
where $M\ge 1$ and $\omega\ge 0$ are such that $\|S(t)_{|_H}\|_{\mathcal{L}(H)}\le Me^{-\omega t}$ for every $t\ge 0$.
Indeed, recalling that
$p(t)=1+(p-1)\|S(t)_{|_H}\|_{\mathcal{L}(H)}^{-2}$, and that $\gamma*\pi$ is a probability measure, the inclusion $L^{p(t)}(E,\gamma*\pi)
    \subseteq L^{p(1+\varepsilon)}(E,\gamma*\pi)$ holds true if $p(t)\ge p(1+\varepsilon)$ or equivalently if
\begin{align*}
    \|S(t)_{|_H}\|_{\mathcal{L}(H)}^2
    \le \frac{p-1}{p(1+\varepsilon)-1}.
\end{align*}
which is clearly satisfied if \eqref{cond_t} is satisfied. 

Therefore, for sufficiently large $t$, Theorem \ref{hyper_appl} yields a sharper inclusion estimate than \cite[Theorem 1.6]{RW03}, whereas for small values of $t$, estimate \eqref{hyp-wang} is more precise.

The hypercontractivity estimate \eqref{hyp-wang} is satisfied in the following example introduced in \cite[Example 3.3]{RW03}.  
\begin{example}\,
    Let $E=\R^n$, $S(t)=e^{-\beta t}{\rm Id}_{\R^n}$, and $\lambda(\xi):=\delta\norm{\xi}^2+2\beta\sum_{i=1}^n\frac{\xi_i^2}{1+\xi_i^2}$, where $\beta,\delta>0$. The generalized Mehler semigroup 
    \begin{align*}
        [P(t)f](x):=\int_{\R^n}f(e^{-\beta t}x+y)d\mu_t(y),\qquad f\in B_b(\R^n),
    \end{align*}
    has a unique invariant measure $\mu$ such that
    $\hat{\mu}(\xi)= \frac{e^{-\delta\|\xi\|^2/(2\beta)}}{(1+\xi_1^2)(1+\xi_2^2)\cdots(1+\xi_n^2)}$. In this case, condition \eqref{condizione_Wang} holds true whenever $(1+\varepsilon)e^{-t\beta}<1$. Thus, estimate \eqref{hyp-wang} holds true for any $p\geq 1$, $t>0$ and $\eps\in(0,e^{\beta t}-1)$.
    Note that for $p>1$ and $t > \frac{1}{\beta}\ln \left(1+\frac{p\varepsilon}{p-1}\right)$, estimate \eqref{idea_davide_semigruppo} is stronger than the hypercontractity one \eqref{hyp-wang}.
\end{example}

\section{Logarithmic Sobolev type inequalities}

In this section we derive some functional inequalities of integral type with respect to invariant measures associated with generalized Mehler semigroups. The occurrence of such inequalities is related to the summability improving property of the semigroup itself. More precisely, we are interested to the case when estimate \eqref{idea_davide_semigruppo} holds true. As it is well known, in the Gaussian case the hypercontractivity of the Ornstein--Uhlenbeck semigroup in the $L^p$-spaces related to its invariant measure $\gamma$, which is a Gaussian measure, is equivalent to the occurrence of the logarithmic Sobolev inequality  
\begin{equation}\label{log-sob}
\int_E |f|^p\ln|f|d\gamma-\left(\int_E|f|^pd\gamma\right) \ln\left(\int_E|f|^p d\gamma\right)\le c_p\int_E|f|^{p-2}\|\nabla_H f\|^2d\gamma
\end{equation}
which holds true for any $p\geq 1$, $f$ regular enough with positive infimum and some $c_p>0$. Here $\nabla_H$ denotes the gradient along the Cameron--Martin space (we refer to \cite{BFFZ24}, for its definition and the main properties). Such inequalities are the counterpart of the Sobolev embeddings which fail in the Gaussian case (see, for example, \cite[Example at page 1074]{Gross75}). Indeed, as it is well known, the inequality \eqref{log-sob} can be interpreted as the embedding of $W^{1,2}(E, \mu)$ in the Orlicz space $L^2\ln L^2(E, \gamma)$, see \cite{Feissner1975,Gross75} for details.

In \cite{AFP23} we derived some modified logarithmic Sobolev inequalities with respect to  invariant measures associated with generalized Mehler semigroups. In such inequalities, differently from the classical case where the entropy of a smooth enough function is estimated by the integral of its gradient as in \eqref{log-sob}, an extra additional term depending on some increments of the function itself appears in the right-hand side. This modification is due to the nonlocal effects of the related generator and provided some integrability property of exponential functions with respect to the invariant measure, in the spirit of the Fernique's theorem, see \cite{BobLed}. However, in such setting we failed to prove even a variant of the classical equivalence result between the hypercontractivity of the semigroup and the modified logarithmic Sobolev inequality obtained.  

Here, we aim at finding a connection between the hypercontractivity type estimate \eqref{idea_davide_semigruppo} which holds true for generalized Mehler semigroups and some logarithmic Sobolev type inequalities with respect to its invariant measure. In this section we assume the following assumptions.

\begin{hyp}\label{Hyp_Mehler}
Let $b\in E$, $Q$ a nonnegative symmetric trace-class operator on $E$ and $\nu$ a Lévy measure on $\mathcal{B}(E)$ such that $\lambda=[b,Q,\nu]$ and let $(S(t))_{t \ge 0}$ be a strongly continuous semigroup on $E$. Assume further that
\begin{enumerate}[(i)] 
\item $A:D(A)\subseteq E\rightarrow E$ is the infinitesimal generator of $(S(t))_{t\geq 0}$ satisfying $\langle A\xi,\xi\rangle<0$ for any $\xi\in D(A)$ with $\xi\neq 0$;
\item there exists an orthonormal basis $\{h_n\,|\, n\in \N\}$ of $E$ consisting of eigenvectors of $A^*$;
\item it holds that
\[\int_{\{\|x\|_E\geq 1\}}\|x\|_E\, d\nu(x)<\infty.\]
\end{enumerate}
\end{hyp}

Note that Hypotheses \ref{Hyp_Mehler} imply that $S(t)$ is exponentially stable and condition in \eqref{log_Levy} is satisfied. Consequently there exists a unique invariant measure for $P(t)$ satisfying Theorem \ref{sunto}.

Following \cite{App07} (see also \cite[Remark 5.11]{PT16}), we say that $f\in C^2_A(E)$ if $f\in C_b(E)\cap C^2(E)$ its first and second order derivatives are uniformly bounded and uniformly continuous on
bounded subsets of $E$, $(Df)(E)\subseteq D(A^*)$ and
$x\mapsto\langle x,A^*Df(x) \rangle\in C_b(E)$. We say $F\in \mathcal{F}C^2_A(E)$ if there exist
$n\in\N$ and $f\in C^2_A(\R^n)$ such that
\begin{align*}
F(x)=f(\gen{x,h_1},\ldots,\gen{x,h_n}),\qquad x\in E.
\end{align*}

\noindent The hypotheses above assure the validity of Hypothesis \ref{Hyp_FR} and consequently the  existence of a unique probability invariant measure (see \cite{CMG96,FR00}) that is the convolution of a Gaussian measure $\gamma$ (with mean $b_\infty$
and covariance operator $Q_\infty$)
and an infinitely divisible (non-Gaussian) probability measure $\pi$ with Lévy measure $\nu_\infty$ defined as in Hypothesis \ref{Hyp_FR}(iii).

Moreover the following result can be proved. See \cite[Theorem 5.2]{App07}, \cite[Section 3]{FR00} and \cite[Remark 5.11]{PT16} for more details.

\begin{thm}\label{4}
The set of functions $\mathcal{F}C^2_A(E)$ is invariant for $P(t)$ and it is a core for $\mathcal{L}$ $($see \eqref{formula_L}$)$ in $L^2(E, \gamma*\pi)$. Moreover
\begin{equation}\label{inv_1}\int_E \mathcal{L}f d(\gamma*\pi)=0, \qquad\;\, f \in \mathcal{F}C^2_A(E).
\end{equation}

\end{thm}

We begin with a preliminary result.

\begin{lemma}
For every $f\in \mathcal{F}C^2_A(E)$ and
every $\Phi\in C^2(\R)$ we have
\begin{align}
\int_E\Phi'(f) & \mathcal{L}f d(\gamma*\pi)\notag\\
&=\int_E\int_E \Big[\Phi(f(x))-\Phi(f(x+y))+\Phi'(f(x))\Big(f(x+y)-f(x)\Big)\Big]d\nu(y)d(\gamma*\pi)(x)\notag\\
&-\frac{1}{2}\int_E\Phi''(f)\langle QDf,Df\rangle d(\gamma*\pi)\label{form_eu}
\end{align}
\end{lemma}

\begin{proof}
First of all, we observe that $\Phi \circ f$ belongs to $\mathcal{F}C^2_A(E)$ and that, by using \eqref{formula_L} it holds that
\begin{align*}
[\mathcal{L}(\Phi\circ f)](x) &=[(\Phi'\circ f)(x)]\left(\frac{1}{2}{\rm Tr}(QD^2f(x))+\langle Ax+b,Df(x)\rangle\right)+ \frac{1}{2}[(\Phi''\circ f)(x)]\langle QDf(x),Df(x)\rangle
\\
&+\int_E\Big[(\Phi\circ f))(x+y)-(\Phi\circ f))(x)-[(\Phi'\circ f)(x)]
\gen{D f(x),y}\chi_{B_1}(y)\Big]d\nu(y).
\end{align*}
Thus, adding and subtracting $\int_E [(\Phi'\circ f)(x)][f(x+y)-f(x)]d\nu(y)$ we get 
\begin{align*}
\mathcal{L}(\Phi\circ f)=&(\Phi'\circ f)\cdot(\mathcal{L}f)+\frac{1}{2}\Phi''(f)\langle Q \nabla f, \nabla f\rangle\notag\\
+&\int_E
\left[(\Phi\circ f)(\cdot+y)-(\Phi\circ f)-(\Phi'\circ f)\cdot\Big(f(\cdot+y)-f\Big)\right]d\nu(y).
\end{align*}
The invariance property in \eqref{inv_1} yields the claim.
\end{proof}

In order to state some logarithmic Sobolev inequalities, it is useful to introduce the notion of \emph{Bregman divergence}. For an arbitrary convex function $\varphi : \mathbb{R} \to \mathbb{R}$ with $\varphi \in C^1(\R)$, the 
Bregman divergence is defined as the first-order Taylor expansion of $\varphi$, i.e., 
\[
\mathrm{D}_\varphi(x,y) := \varphi(x) - \varphi(y) - \varphi'(y)(x-y), \qquad x,y \in \mathbb{R}.
\]
It is easy to see that $\mathrm{D}_\varphi \ge 0$ in $\R^2$ thanks to the convexity of $\varphi$. 

Bregman divergences are widely used in statistical learning and its applications (see, e.g., \cite{Amari16,AJLS17,NN09}), and they are also closely related to entropy theory (see \cite{Chafai04,Wang14,Wang14_nonlocal}).

We consider two different types of Bregman divergence: the first one obtained when $\varphi(x) = |x|^p$, $x\in \R$ and $p>1$. In this case
\begin{align*}
    \mathrm{D}_{|\cdot|^p}(x,y) = |x|^p - |y|^p - p\,|y|^{p-1}\operatorname{sgn}(y)\,(x-y), \qquad x,y \in \mathbb{R}.
\end{align*}
Note that $\mathrm{D}_{|\cdot|^2}(x,y) = (y-x)^2$.  
Another interesting Bregman divergence is the so-called \emph{Kullback--Leibler divergence} defined in correspondence of the function $\varphi(x)=x \ln x$, $x>0$. In this case, to simplify the notation, we denote $\mathrm{D}_\varphi$ by $\mathrm{KL}$ that is defined as follows
\begin{align*}
    \mathrm{KL}(x,y) :=y-x+x(\ln x-\ln y), \qquad x,y>0.
\end{align*}

Before going on we set
\begin{equation}\label{q*0}
q_0:=(D_+\|S(\cdot)_{|_H}\|_{\mathcal{L}(H)}^{-2})(0)=\liminf_{h\ra0^+}\frac{1}{h}\Bigg(\frac{1}{\|S(h)_{|_H}\|^{2}_{\mathcal{L}(H)}}-1\Bigg)
\end{equation}
where $D_+$ denotes the right-Dini derivatives, see \eqref{Dini}.
We prove two different inequalities assuming that $q_0 \in (0,\infty)$. In view of this, in the following lemma we provide a sufficient condition that guarantee that such assumption holds true.

\begin{lemma}
Assume that the semigroup $(S(t))_{t \ge 0}$ satisfies the estimate $\|S(t)_{|_H}\|_{\mathcal{L}(H)}\le e^{-\omega t}$
for any $t>0$ and some positive $\omega >0$. If, further, the growth bound $\omega_H$ of the semigroup $(S(t)_{|_H})_{t\geq 0}$ defined as
\begin{equation*}
    \omega_H: = \inf \left\{ \omega \in \mathbb{R} \;\middle|\; \exists M_\omega \geq 1 \text{ such that } \|S(t)_{|_H}\|_{\mathcal{L}(H)} \leq M_\omega e^{\omega t} \text{ for all } t \geq 0 \right\}
\end{equation*}
is not $-\infty$, then $q_0 \in (0,\infty)$.
\end{lemma}
\begin{proof}
The proof is based on the fact that 
\begin{align*}
  e^{\omega_Ht}\le   \|S(t)_{|_H}\|_{\mathcal{L}(H)} \leq e^{-\omega t}, \qquad\;\,  t \ge 0,
\end{align*}
(see \cite[Corollary 1.4 and Proposition 2.2 of Chapter IV]{EN00} for details).
Consequently,
$$\frac{e^{2\omega t}-1}{t}\le \frac{1}{t}\Bigg(\frac{1}{\|S(t)_{|_H}\|_{\mathcal{L}(H)}^2}-1\Bigg)\le\frac{e^{-2\omega_H t}-1}{t}$$
for any $t>0$, whence $2\omega \leq [D_+\|S(t)_{|_H}\|^{-2}](0) \leq -2\omega_H.$
\end{proof}

For a measure space $(\Omega,\mathcal{F},\mu)$ and a positive $\mathcal{F}$-measurable function $\psi:\Omega\ra\R$ we set
\begin{align*}
    \mathrm{Ent}_\mu(\psi):=\int_\Omega \psi\ln\psi d\mu-\left(\int_\Omega\psi d\mu\right)\ln\left(\int_\Omega\psi d\mu\right).
\end{align*}
We can now state and prove one of the main result of this section.

\begin{thm}\label{logSobolev_Strana}
Let assume that $q_0 \in (0,\infty)$. Then, for every $r\in [2,\infty)$ and $f \in \mathcal{F}C^2_A(E)$ the following inequality holds true 
\begin{align}\label{claim_Sob_1}
\int_E {\rm Ent}_\gamma| f(\cdot+y)|^rd\pi(y)\le&
\frac{r^2}{2q_0}\int_E|f|^{r-2}\langle Q Df,Df\rangle d(\gamma*\pi)\notag\\
&+\frac{r}{(r-1)q_0}\int_E\int_E \mathrm{D}_{|\cdot|^r}(f(x+y),f(x))d\nu(y)d(\gamma*\pi)(x).
\end{align}
\begin{proof}
We split the proof into two steps.

{\em Step 1.} We start considering a smooth function $f\in \mathcal{F}C^2_A(E)$ with positive infimum in $E$ and $r \in (1,\infty)$. Then, thanks to Theorem \ref{4} and the representation \eqref{generalized_mehler}, $P(t)f$ belongs to $\mathcal{F}C^2_A$ and has positive infimum too. Moreover $t\mapsto P(t)f$ is differentiable in $(0,\infty)$ and $\frac{d}{dt}P(t)f=\mathcal{L}P(t)f$ where the action of $\mathcal{L}$ on $P(t)f$ is defined in \eqref{formula_L}.
Estimate \eqref{idea_davide_semigruppo} yields that
$\|P(t)f\|_{r(t), r}\le \|f\|_{r}$ where $r(t):=1+(r-1)\|S(t)_{|_H}\|_{\mathcal{L}(H)}^{-2}$. Consequently, set $F_r(t):=\|P(t) f\|_{r(t),r}$ for every $t>0$, it follows that $(D^+F_r)(0)\le 0$.
Since 
$$F_r(t)=\left(\int_E\Big(\int_E |P(t)f(x+y)|^{r(t)}d\gamma(x)\Big)^{\frac{r}{r(t)}}d\pi(y)\right)^{\frac{1}{r}}$$
and, by Proposition \ref{Proposition_Dini_derivatives},
$$\Bigg(D^+ \frac{1}{r(\cdot)}\Bigg)(t)=-\frac{(r-1)(D_+\|S(\cdot)_{|_H}\|_{\mathcal{L}(H)}
^{-2})(t)}{r^2(t)}=:-\frac{(r-1)q^*(t)}{r^2(t)}$$
we get that
\begin{align*}
(D^+F_r)(t)&=
\frac{1}{r}\|P(t)f\|_{r(t),r}^{1-r}\int_E\left(\int_E |P(t)f(x+y)|^{r(t)}d\gamma(x)\right)^{\frac{r}{r(t)}-1}\\
&\times\bigg[r (r-1) \frac{q^*(t)}{r^2(t)}{\rm Ent}_\gamma(|P(t)f(\cdot+y)|^{r(t)})\\
&+r\int_E |P(t)f(x+y)|^{r(t)-2}(P(t)f)(x+y)\mathcal{L}(P(t)f)(x+y)d\gamma(x)\bigg]d\pi(y).
\end{align*}
Taking into account that $r(0)=r$ and $q^*(0)=q_0$, the inequality $(D^+F_r)(0)\le 0$ is equivalent to the estimate 
\begin{equation}\label{cons_dip}
\int_E{\rm Ent}_\gamma |f(\cdot+y)|^r d \pi(y)\le -\frac{r^2}{(r-1)q_0}\int_E |f(x+y)|^{r-2}f(x+y)(\mathcal{L}f)(x+y)d\gamma(x)d\pi(y).
\end{equation}
Now, using the equality \eqref{form_eu} with $\Phi(x)=|x|^r$, we deduce that the integral in the right-hand side of \eqref{cons_dip} equals to
\begin{align}\label{appl_lemma}
\int_E\left[\frac{1}{r}\int_E\Big(|f|^r-|f(\cdot+y)|^r+r|f|^{r-2}f\big(f(\cdot+y)-f\big) \Big)d\nu(y)-\frac{1}{2}(r-1)|f|^{r-2}\langle Q Df,Df\rangle\right] d(\gamma*\pi).
\end{align}
Noticing that
$$
(|f(x)|^r-|f(x+y)|^r+r|f(x)|^{r-2}f(x)\big(f(x+y)-f\big)= \mathrm{D}_{|\cdot|^r}(|f(x+y)|,|f(x)|)
$$
and replacing \eqref{appl_lemma} in \eqref{cons_dip} we get inequality \eqref{claim_Sob_1}.\\
{\em Step 2.} Here, we remove the assumption on the positivity  on the infimum of $f$, arguing by approximation.
Let assume that $f\in \mathcal {F}C^2_A(E)$ and define
$f_n:=(f^2+n^{-1})^{1/2}$. By the first part of the proof we have
\begin{align*}
\int_E {\rm Ent}_\gamma| f_n(\cdot+y)|^rd\pi(y)\le& \frac{r}{(r-1)q_0}\int_E\int_E \mathrm{D}_{|\cdot|^r}(|f_n(x+y)|,|f_n(x)|)d\nu(y)d(\gamma*\pi)(x)\notag\\
&+\frac{r^2}{2q_0}\int_E|f_n|^{r-2}\langle Q Df_n,Df_n\rangle d(\gamma*\pi).
\end{align*}
Since $0<f_n^r\le \|f^2+\frac{1}{n}\|_\infty^{r/2}$ and $|\nabla f_n|^2\le |\nabla f|^2$ for any $n \in \N$, by the dominated convergence theorem we deduce estimate \eqref{claim_Sob_1} for functions in $\mathcal{F}C^2_A(E)$.
\end{proof}
\end{thm}

\begin{rmk}{\rm As Step 1 in the proof of Theorem \ref{logSobolev_Strana} shows, estimate \eqref{claim_Sob_1} holds true for every $r\in (1,\infty)$ provided that $f$ has positive infimum. 

In the classical logarithmic Sobolev inequality, the right-hand side of estimate \eqref{claim_Sob_1} contains only the first term and an argument based on the monotone convergence theorem allows to extend the previous inequality to any smooth function even for $r \in (1,2)$. Here, the presence of the additional term containing the Bregman divergence seems to be a technical obstruction to the extension of the results for $r \in (1,2)$, in the general case.
}
\end{rmk}

\begin{rmk}{\rm In the classical case, the logarithmic Sobolev inequality related to the invariant measure of a semigroup is written in terms of the entropy of a function with respect to such a measure. In our case, even if the invariant measure for $P(t)$ is given by the convolution $\mu:=\gamma*\pi$, the logarithmic Sobolev inequality proved in Theorem \ref{logSobolev_Strana} does not involve the entropy of a function with respect to $\mu$ but is given in terms of the integral with respect to $\pi$ of the entropy with respect to $\gamma$. Moreover, we cannot deduce the classical version of the logarithmic Sobolev inequality since
$$\int_E {\rm Ent}_\gamma| f(\cdot+y)|d\pi(y)\le {\rm Ent}_{\gamma*\pi}(| f|)$$
for any $f\in B_b(E)$. 
Indeed, by the convexity of the function $(0,\infty)\ni x\mapsto x\ln x$ and the finiteness of the measure $\pi$, for every $f\in B_b(E)$ it holds that
$$\left(\int_E |f|d(\gamma *\pi)\right)\ln\left( \int_E|f|d(\gamma*\pi)\right)\le \int_E\left(\int_E|f(\cdot+y)|d\gamma\ln\int_E|f(\cdot+y)|d\gamma\right)d\pi(y).$$
}
\end{rmk}

To go further, we recall that the hypercontractivity of the classical Ornstein--Uhlenbeck semigroup $R(t)$  is equivalent to what is called \emph{weak hypercontractivity} in \cite{NPY20} which states that for every $f \in L^1(E,\gamma)$ with $e^f \in L^1(E,\gamma)$ and for every $q \in [1,e^{2t}]$, it holds
\begin{align}\label{weak_hyp_gauss}
    \|e^{R(t)f}\|_{L^q(E,\gamma)} \le \|e^f\|_{L^1(E,\gamma)},
\end{align}
where $\gamma$ is the unique invariant measure for $R(t)$, which is of Gaussian type. A proof of such result can be found in \cite[Proposition 4]{BakryEmery}. Inequalities of type \eqref{weak_hyp_gauss} have also been studied in connection with purely jump processes and $\Phi$-entropy inequalities (see, e.g., \cite{BobLed,BT06,Hariya18}).  

We will show that a weak hypercontractivity type estimate holds in our setting as well, and that it implies a logarithmic Sobolev inequality formulated in terms of the  Kullback--Leibler divergence.

\begin{thm}\label{thm_logSobolev_KL}
Assume that $q_0 \in (0,\infty)$ $($see \eqref{q*0}$)$. Then the following statements hold true:
\begin{enumerate}[(i)]
    \item\label{1} for every $f \in L^1(E,\gamma*\pi)$ with $e^f \in L^1(E,\gamma*\pi)$ and for every $q \in [1, \|S(t)_{|_H}\|_{\mathcal{L}(H)}^{-2}]$, it holds that
    \begin{align}\label{weak_hyp_convoluted}
        \| e^{P(t)f} \|_{q,1} \le \| e^f \|_{L^1(E,\gamma*\pi)};
    \end{align}

    \item 
    for every $f \in \mathcal{F}C_A^2(E)$, we have
    \begin{align}\label{logSobolev_KL}
        \int_E \mathrm{Ent}_\gamma\big[e^{f(\cdot+y)}\big]\, d\pi(y) 
        \le & \frac{1}{2q_0} \int_Ee^{f(x)}\|\nabla_H f(x)\|^2_H d(\gamma*\pi)(x)\notag\\
        &+ \frac{1}{q_0}\int_E\int_E \mathrm{KL}(e^{f(x)},e^{f(x+y)})d\nu(y)d(\gamma*\pi)(x).
    \end{align}
    \end{enumerate}
\end{thm}

\begin{proof} (i) Let $f \in L^1(E,\gamma*\pi)$ with $e^f \in L^1(E,\gamma*\pi)$ and let $r>1$. Observe that, by Jensen's inequality, we have
\begin{align}
    [P(t)e^{\frac{1}{r}f}](x) 
    &= \int_E \mathrm{exp}\left(\frac{1}{r}f(S(t)x+y)\right) d\mu_t(y) \notag \\
    &\ge \mathrm{exp}\left(\frac{1}{r}\int_E f(S(t)x+y) d\mu_t(y)\right)
    = e^{\frac{1}{r}[P(t)f](x)}, \qquad\; \, x \in E. \label{Jensen_exp}
\end{align}
Now, let $r(t) = 1 + (r-1) \|S(t)_{|_H}\|^{-2}_{\mathcal{L}(H)}$, and observe that, by \eqref{idea_davide_semigruppo}, we have for every $t>0$
\begin{align}
    \| P(t) e^{\frac{1}{r}f} \|_{r(t),r} 
    \le \| e^{\frac{1}{r}f} \|_{L^r(E,\gamma*\pi)}
    = \| e^f \|^{\frac{1}{r}}_{L^1(E,\gamma*\pi)}. \label{norm_exp}
\end{align}
Combining \eqref{Jensen_exp} and \eqref{norm_exp}, we obtain
\begin{align}\label{exp_Pt}
    \| e^{\frac{1}{r}P(t)f} \|^r_{r(t),r} \le \| e^f \|_{L^1(E,\gamma*\pi)}.
\end{align}
Finally, observe that for every $t>0$ it holds 
$\lim_{r \to +\infty} \frac{r(t)}{r} = \|S(t)_{|_H}\|_{\mathcal{L}(H)}^{-2}$, and therefore, by \eqref{exp_Pt}, we get
\begin{align*}
    \| e^{P(t)f} \|_{\|S(t)_{|_H}\|_{\mathcal{L}(H)}^{-2},1} 
    = \lim_{r \to +\infty} \| e^{\frac{1}{r}P(t)f} \|^r_{r(t),r} 
    \le \| e^f \|_{L^1(E,\gamma*\pi)}
\end{align*}
whence the claim.

(ii) We set $q(t)=\|S(t)_{|_H}\|^{-2}_{\mathcal{L}(H)}$ $t>0$ and  consider the function
\begin{align*}
    F(t):=\| e^{P(t)f} \|_{q(t),1}, \quad\; t>0.
\end{align*}
Long but straightforward calculations give
\begin{align*}
     (D^+F)(t)  =\int_E & \left(\int_E e^{q(t)[P(t)f](x+y)}d\gamma(x)\right)^{\frac{1}{q(t)}-1}\\
    &\times\bigg(-\frac{q^*(t)}{q^2(t)}\left(\int_E e^{q(t)[P(t)f](x+y)}d\gamma(x)\right)\ln\int_Ee^{q(t)[P(t)f](x+y)}d\gamma(x)\\
    & +\frac{q^*(t)}{q(t)}\int_E e^{q(t)[P(t)f](x+y)}[P(t)f](x+y)d\gamma(x)\\
    &+\frac{1}{q(t)}\int_Ee^{q(t)[P(t)f](x+y)}(\mathcal{L}P(t)f)(x+y)d\gamma(x)\bigg)d\pi(y),
\end{align*}
where $q^*(t)= (D_+ q)(t)$. Taking into account that $q(0)=1$ and that $q^*(0)=q_0$, we deduce
\begin{align*}
    (D^+F)(0) &=-q_0\int_E\left[\left(\int_Ee^{f(x+y)}d\gamma(x)\right)\ln\int_Ee^{f(x+y)}d\gamma(x)\right]d\pi(y)\\
    &+ q_0\int_E\int_Ee^{f(x+y)}f(x+y)d\gamma(x)d\pi(y)+\int_E\int_Ee^{f(x+y)}(\mathcal{L}f)(x+y)d\gamma(x)d\pi(y).
\end{align*}
Equality \eqref{form_eu} with $\Phi(x)=e^x$ yields
\begin{align*}
\int_E\int_Ee^{f(x+y)} & (\mathcal{L}f) (x+y)d\gamma(x)d\pi(y) =-\frac{1}{2}\int_Ee^{f(x)}\|\nabla_H f(x)\|^2_Hd(\gamma*\pi)(x)\\
& +\int_E\int_E e^{f(x)}-e^{f(x+y)}+e^{f(x)}(f(x+y)-f(x))d\nu(y)d(\gamma*\pi)(x)\\
&=-\frac{1}{2}\int_Ee^{f(x)}\|\nabla_H f(x)\|^2_H d(\gamma*\pi)(x)- \int_E\int_E \mathrm{KL}(e^{f(x)},e^{f(x+y)})d\nu(y)d(\gamma*\pi)(x).
\end{align*}
Now, thanks to property \eqref{1} $(D^+F)(0)\leq 0$ and recalling that $q_0$ is positive, we get \eqref{logSobolev_KL}. 
\end{proof}

\begin{rmk}
If $g$ belongs to $\mathcal{F}C^2_A(E)$ and it has positive infimum, then
taking $f=r\ln g$ in \eqref{logSobolev_KL} we get
\begin{align}
    \int_E \mathrm{Ent}_\gamma\big[g^r(\cdot+y)\big]\, d\pi(y) 
        \le & \frac{r^2}{2q_0} \int_E  g^{r-2} \|\nabla_H g\|_H^2 \, d(\gamma*\pi)\notag\\
        &+ \frac{1}{q_0}\int_E \int_E \mathrm{KL}\big(g^r(x), g^r(x+y)\big)\, d\nu(y)\, d(\gamma*\pi)(x).\label{logSobolev_KL_infpositivo}
\end{align}
This inequality is similar to \eqref{claim_Sob_1}, with the Bregman divergence $\mathrm{D}_{|\cdot|^r}$ replaced with the Kullback--Leiber divergence and coincides with the classical logarithmic Sobolev inequality when $\pi= \delta_0$.
Actually, in this latter case, inequalities \eqref{hyp_classica}, \eqref{log-sob} and \eqref{weak_hyp_gauss} are equivalent to each other (see \cite[Proposition 4]{BakryEmery}) and also to the inequality
\begin{align*}
        \mathrm{Ent}_\gamma[e^{f}] 
        \le & c \int_Ee^{f}\|\nabla_H f\|^2_H d\gamma\notag
    \end{align*}
which coincides with the estimate \eqref{logSobolev_KL} in the purely Gaussian case, (see also  \cite[Theorem 7.1]{BT06}).
In our case, the lack of symmetry of the operator $\mathcal L$ does not allow us to use the same arguments to show the same equivalences. Moreover, it is not difficult to show that
\begin{align*}
\mathrm{D}_{|\cdot|^p}(x,y)\le \mathrm{KL}(y^p,x^p),
\end{align*}
for any $x,y>0$. Therefore estimate \eqref{claim_Sob_1} is stronger than \eqref{logSobolev_KL_infpositivo} when they hold. However in general, it can be proved that there exists $t_0<0$ such that 
$$\mathrm{D}_{|\cdot|^p}(x,y)\le \mathrm{KL}(y^p,x^p)$$
holds true for $t_0<\frac{y}{x}<0$.
Consequently, in the general case, estimates \eqref{claim_Sob_1} and \eqref{logSobolev_KL_infpositivo} cannot be compared.
\end{rmk}

\section*{Declarations}

\subsection*{Acknowledgments} 
The authors wish to thank D. Pallara (Salento), D. Addona (Parma) and D. Bignamini (Insubria) for valuable discussions and the suggestion of some useful references. 

\subsection*{Fundings} The authors are members of GNAMPA (Gruppo Nazionale per l’Analisi Matematica, la Probabilit\`a
e le loro Applicazioni) of the Italian Istituto Nazionale di Alta Matematica (INdAM). The authors were partially supported by the INdAM-GNAMPA Project, codice CUP \#E5324001950001\#.


\subsection*{Data Availability} Data sharing not applicable to this article as no datasets were generated or analysed during the current study.

\end{document}